\documentclass[english]{article}
\usepackage[T1]{fontenc}
\usepackage[latin9]{inputenc}
\usepackage{geometry}
\usepackage{amsmath}
\usepackage{amsthm}
\usepackage{thmtools,thm-restate}
\usepackage{amssymb}

\usepackage{libertine}
\usepackage[libertine]{newtxmath}

\makeatletter
\numberwithin{equation}{section}
\numberwithin{figure}{section}
\theoremstyle{plain}
\newtheorem{thm}{\protect\theoremname}
  \theoremstyle{plain}
  \newtheorem{conjecture}[thm]{\protect\conjecturename}
  \theoremstyle{plain}
  
  \theoremstyle{plain}
  \newtheorem{lem}[thm]{\protect\lemmaname}
  \theoremstyle{plain}
  \newtheorem{claim}[thm]{Claim}
  \theoremstyle{plain}
  \newtheorem{remark}[thm]{Remark}
  \theoremstyle{definition}
  \newtheorem{defn}[thm]{\protect\definitionname}
  \theoremstyle{remark}
  \newtheorem*{acknowledgement*}{\protect\acknowledgementname}
  \theoremstyle{plain}
  \newtheorem*{assumption*}{\protect\assumptionname}
  \theoremstyle{remark}
  \newtheorem*{rem*}{\protect\remarkname}

\newenvironment{proofof}[1]{\smallskip\noindent{\bf Proof of #1.}}%
        {\hspace*{\fill}$\Box$\par}

\usepackage{hyperref}
\hypersetup{
    colorlinks,
    allcolors=red
}
\usepackage[lined,boxed,ruled,norelsize,algo2e]{algorithm2e}
\@addtoreset{section}{part}

\makeatother

\usepackage{babel}
  \providecommand{\acknowledgementname}{Acknowledgement}
  \providecommand{\assumptionname}{Assumption}
  \providecommand{\conjecturename}{Conjecture}
  \providecommand{\corollaryname}{Corollary}
  \providecommand{\definitionname}{Definition}
  \providecommand{\lemmaname}{Lemma}
  \providecommand{\remarkname}{Remark}
\providecommand{\theoremname}{Theorem}

\begin{document}
\global\long\def\defeq{\stackrel{\mathrm{{\scriptscriptstyle def}}}{=}}
\global\long\def\norm#1{\left\Vert #1\right\Vert }
\global\long\def\R{\mathbb{R}}
 \global\long\def\Rn{\mathbb{R}^{n}}
\global\long\def\tr{\mathrm{Tr}}
\global\long\def\diag{\mathrm{diag}}
\global\long\def\cov{\mathrm{Cov}}
\global\long\def\E{\mathbb{E}}
\global\long\def\P{\mathbb{P}}
\global\long\def\Var{\mathrm{Var}}
\global\long\def\rank{\mathrm{rank}}
\global\long\def\lref#1{\text{Lem }\ltexref{#1}}
\global\long\def\lreff#1#2{\text{Lem }\ltexref{#1}.\ltexref{#1#2}}
\global\long\def\ltexref#1{\ref{lem:#1}}\global\long\def\ttag#1{\tag{#1}}
\global\long\def\cirt#1{\raisebox{.5pt}{\textcircled{\raisebox{-.9pt}{#1}}}}

\newcommand{\dd}{d}
\newcommand{\op}{\ensuremath{\mathrm{op}}\xspace}
\newcommand{\p}{\ensuremath{\mathbb{P}}\xspace}
\newcommand{\Gauss}{\ensuremath{\mathcal{N}}\xspace}
\newcommand{\spe}{\ensuremath{\mathrm{op}}\xspace}
\newcommand{\Def}{\ensuremath{\mathrm{def}}\xspace}
\newcommand{\dtv}{\ensuremath{d_\mathrm{TV}}\xspace}
\newcommand{\eps}{\epsilon}
\newcommand{\Ind}{\ensuremath{\mathbf{1}}}

\newcommand{\eat}[1]{}

\newcommand{\haotian}[1]{{\color{blue} \textbf{Haotian:} #1}}
\newcommand{\santosh}[1]{{\color{green} \textbf{santosh:} #1}}
\newcommand{\yintat}[1]{{\color{red} \textbf{Yin-Tat:} #1}}
\newcommand{\comment}[1]{}

\title{A Generalized Central Limit Conjecture for Convex Bodies}
\author{Haotian Jiang\thanks{University of Washington, jhtdavid@uw.edu.}, \ 
Yin Tat Lee\thanks{University of Washington and Microsoft Research, yintat@uw.edu. Research supported in part by NSF Awards CCF-1740551, CCF-1749609, and DMS-1839116.}, \ Santosh S. Vempala\thanks{Georgia Tech, vempala@gatech.edu. Research supported in part by NSF Awards CCF-1563838, CCF-1717349, DMS-1839323 and E2CDA-1640081.}}
\date{}

\maketitle
\begin{abstract}
The central limit theorem for convex bodies says that with high probability the marginal of an isotropic log-concave distribution along a random direction is close to a Gaussian, with the quantitative difference determined asymptotically by the Cheeger/Poincare/KLS constant. Here we propose a generalized CLT for marginals along random directions drawn from any isotropic log-concave distribution; namely, for $x,y$ drawn independently from isotropic log-concave densities $p,q$, the random variable $\langle x,y\rangle$ is close to Gaussian. Our main result is that this generalized CLT is quantitatively equivalent (up to a small factor) to the KLS conjecture. Any polynomial improvement in the current KLS bound of $n^{1/4}$ in $\R^n$ implies the generalized CLT, and vice versa. This tight connection suggests that the generalized CLT might provide insight into basic open questions in asymptotic convex geometry.
\end{abstract}

\section{Introduction}
Convex bodies in high dimensions exhibit surprising asymptotic properties, i.e., phenomena that become sharper as the dimension increases. As an elementary example, most of the measure of a sphere or ball in $\R^n$ lies within distance $O(1/\sqrt{n})$ of any bisecting hyperplane, and a one-dimensional marginal is close to a Gaussian, i.e., its total variation distance to a Gaussian of the same variance is $O(1/\sqrt{n})$. 
A striking generalization of this is the central limit theorem for convex bodies in Theorem~\ref{thm:CLTConvexBodies}, originally due to Klartag~\cite{Klartag2007}.
A function $h:\Rn \rightarrow \R_{+}$ is called {\em log-concave}
if it takes the form $h = \exp(- f)$ for a convex function $f: \Rn \rightarrow \R \cup \{ \infty \}$.
A probability measure is log-concave if it has a log-concave density.  
A measure is said to be {\em isotropic} if it has zero mean and identity covariance. 

\begin{thm}[Central Limit Theorem]
\label{thm:CLTConvexBodies}
Let $p$ be an isotropic log-concave measure in $\Rn$ and $y\sim p$.
Then we have
\[
\P_{x \sim S^{n-1}} \left[ \dtv \left( \langle x,y \rangle , \Gauss(0,1)  \right) \geq c_n \right] \leq c_n,
\]
for some constants $c_n$ that tends to $0$ as $n\rightarrow +\infty$.
\end{thm}


The central limit theorem is closely related to the {\em thin-shell} conjecture (also known as the {\em variance hypothesis})~\cite{ABP03,BK03}.
Let $\sigma_n \geq 0$ satisfy 
\[
\sigma_n^2 = \sup_p \E_{x \sim p}\left[ \left(\norm{x} - \sqrt{n} \right)^2 \right],
\]
where the supremum is taken over all isotropic, log-concave measures $p$ in $\Rn$.
The thin-shell conjecture~\cite{ABP03,BK03} asserts the existence of a universal constant $C$ such that $\sigma_n^2 < C$
for all $n \in \mathbb{N}$. It is closely connected to the CLT: by a direct calculation, the CLT implies a bound on $\sigma_n$ (and the conjectured CLT parameter implies the thin-shell conjecture);
Moreover, $c_n = O(\sigma_n \log n/\sqrt{n})$~\cite{ABP03,Eldan2013}. 
The first non-trivial bound on $\sigma_n$, which gives the first non-trivial bound on $c_n$ in Theorem~\ref{thm:CLTConvexBodies}, was due to Klartag~\cite{Klartag2007}. This was followed by several improvements and refinements \cite{Paouris2006, Klartag2007b, Fleury2010, GuedonM11}. 
The current best bound is $\sigma_n = O(n^{1/4})$ which implies $c_n = O(n^{-1/4} \log n)$ \cite{lee2016arxiv}. This follows from the well-known fact that $\sigma_n = O( \psi_n)$, where $\psi_n$ is the KLS constant (also known as the inverse Cheeger constant) defined as follows.
\begin{defn}[KLS constant]
For a log-concave density $p$ in $\R^n$ with induced measure $\mu_p$, the KLS constant $\psi_p$ is defined as
\[
\frac{1}{\psi_p} = \inf_{S\subset \R^n, \mu_p(S)\le 1/2} \frac{\mu_p(\partial S)}{\mu_p(S)}.
\]
We define $\psi_n$ be the supremum of $\psi_p$ over all isotropic log-concave densities $p$ in $\R^n$.
\end{defn} 

\begin{thm}[\cite{lee2016arxiv}]
The KLS constant of any isotropic log-concave density in $\R^n$ is $O(n^{1/4})$. 
\end{thm}

For other connections and implications of the KLS conjecture, including its equivalence to spectral gap and its implication of the slicing conjecture, the reader is referred to recent surveys \cite{guedon2014concentration,LVSurvey18} and this comprehensive book \cite{brazitikos2014geometry}.

A key fact used in the above theorem is the following elementary lemma about log-concave densities. 

\begin{lem}[Third moment]\label{lem:third}
For $x,y$ drawn independently from an isotropic log-concave density $p$, we have $\E(\langle x,y\rangle^3) = O(n^{1.5})$. 
\end{lem}
We remark that the third moment bound in Lemma~\ref{lem:third}, holds even if $x,y$ are drawn independently from different measures. 

If the KLS conjecture is true, then the expression above is $O(n)$. 
It is shown in an earlier version of \cite{lee2016arxiv} that any polynomial improvement in the third moment bound to $n^{1.5-\eps}$ for some $\eps > 0$ would lead to an improvement in the bound on the KLS constant to $n^{1/4 - \eps'}$ for some $\eps'>0$. 
(The techniques used in the corresponding part of the preprint~\cite{lee2016arxiv} are formally included in this paper.) 

Motivated by the above connection, we propose a generalized CLT in this paper. 
To formally state our geralized CLT, we need the definition of $L_p$ Wasserstein distance.
\begin{defn}[$L_p$ Wasserstein distance or $W_p$ distance]\label{def:Wasser}
The $L_p$ Wasserstein distance between two probability measures $\mu$ and $\nu$ in $\R$ for $p \geq 1$ is defined by 
\[
W_p(\mu,\nu) \overset{\Def}{=} \inf_{\pi} \left[ \int |x-y|^p \dd \pi(x,y) \right]^{\frac{1}{p}},
\]
where the infimum is over all couplings of $\mu$ and $\nu$, i.e. probability measures $\pi$ in $\R^{2}$ that have marginals $\mu$ and $\nu$. 
\end{defn}
When convenient we will denote $W_p(\mu,\nu)$ also be $W_p(x,y)$ where $x\sim \mu, y \sim \nu$.
Our generalized CLT is stated using the $W_2$ distance, which is a natural choice, also used in related work on CLT's \cite{zhai2018high,eldan2018clt}.

The content of the conjecture is that one can replace the uniform distribution on the sphere (or Gaussian) with any isotropic log-concave density, i.e., along most directions with respect to any isotropic log-concave measure, the marginal of an isotropic log-concave measure is approximately Gaussian.
\begin{conjecture}[Generalized CLT]
\label{conj:GenCLT}
Let $x,y$ be independent random vectors drawn from isotropic log-concave densities $p, q$ respectively and $G\sim \Gauss(0,n)$. Then, 
\begin{equation}
\label{eqn:GenCLTconj}
W_2(\langle x,y \rangle, G) = O(1).
\end{equation}
\end{conjecture}
The current best upper bound on the $W_2$ distance in Equation~(\ref{eqn:GenCLTconj}) is the trivial bound of $O(\sqrt{n})$. As we will see later, a third moment bound of order $O(n)$ in Lemma \ref{lem:third} would be implied if Conjecture~\ref{conj:GenCLT} holds.

Our main result is that this Generalized CLT is equivalent (up to a small factor) to the KLS conjecture, and any polynomial improvement in one leads to a similar improvement in the other. 
\begin{thm}[Generalized CLT equivalent to KLS]
\label{thm:GCLTEquivKLS}
Fix $\eps \in (0,1/2)$. If for every isotropic log-concave measure $p$ in $\Rn$ and independent vectors $x,y\sim p$ and $g\sim \Gauss(0,n)$, we have $W_2(\langle x,y \rangle, g) = O \left(n^{1/2-\eps} \right)$,
then for any $\delta > 0$, we have $\psi_n = O \left(n^{1/4 - \eps/2 + \delta} \right)$.

On the other hand, if we have $\psi_n = O\left(n^{1/4 - \eps/2} \right)$, then for any isotropic log-concave measures $p,q$ in $\Rn$, independent vectors $x\sim p, y\sim q$ and $\delta>0$, we have $W_2(\langle x,y \rangle, G) = O \left(n^{1/2-\eps+\delta} \right)$.
\end{thm}

\begin{remark} 
We emphasize that the equivalence between Generalized CLT and the KLS conjecture in Theorem~\ref{thm:GCLTEquivKLS} does not hold in a pointwise sense, i.e. the Generalized CLT for a specific isotropic log-concave measure $p$ in $\Rn$ alone does not imply the corresponding bound for $\psi_p$ and vice versa.
One needs to establish the Generalized CLT for all isotropic log-concave measures in $\Rn$ in order to deduce the KLS conjecture.
\end{remark}

The proof of Theorem~\ref{thm:GCLTEquivKLS} proceeds in three steps: (1) in Theorem~\ref{thm:TMBtoKLS} below, we show that an improved third moment bound implies an improved bound on the KLS constant (an earlier version of this part of the proof is implicit in the preprint \cite{lee2016arxiv}), (2) in Theorem~\ref{thm:GenCLTtoTMB}, we show that an improved bound for Generalized CLT implies an improved third moment bound, and (3) in Theorem~\ref{thm:KLStoGenCLT}, we show that an improved bound on the KLS constant implies an improved bound for Generalized CLT. While all three parts are new and unpublished (except on the arXiv), the proof of (3) is via a coupling with Brownian motion (we discuss the similarity to existing literature \cite{eldan2018clt}), (2) is relatively straightforward, and (1) is the most technical, based on a carefully chosen potential function and several properties of an associated tensor. 

The main intermediate result in our proof that the Generalized CLT implies the KLS conjecture is the following theorem.
\begin{restatable}{thm}{TMBtoKLS}
\label{thm:TMBtoKLS}
Fix $\epsilon \in (0,1/2)$. If for every isotropic log-concave distribution $p$ in $\Rn$ and independent vectors $x,y \sim p$, we have
\begin{eqnarray}
\label{eqn:ThirdMomentBound}
\E_{x,y \sim p} \left(\langle x,y\rangle^3 \right) = O\left(n^{1.5-\epsilon} \right),
\end{eqnarray}
then for any $\delta > 0$, we have $\psi_n = O\left(n^{1/4-\epsilon/2 + \delta} \right)$.
\end{restatable}

In fact what we show is that the KLS constant $\psi_n$ can be bounded in terms of the third moment.
\begin{thm}\label{thm:refinedEldan}
Let $p$ range over all isotropic log-concave distributions in $\Rn$. Then,
\begin{align}
\label{eqn:BoundKLSbyTMB}
\psi_n^2 \leq \frac{\widetilde{O}\left(1\right)}{n} \cdot \sup_{p} \E_{x,y\sim p} \left(\langle x,y \rangle^3 \right) = \widetilde{O}(1)\cdot \sup_p \E_{\theta \sim S^{n-1}} \norm{\E_{x\sim p}\left( \langle x,\theta \rangle xx^T  \right) }_F.
\end{align}
\end{thm}
This intermediate result might be of independent interest and is in fact a refinement of the following bound on the KLS constant given by Eldan~\cite{Eldan2013}.
\begin{align}
\label{eqn:EldThinShellBound}
\psi_n^2 \leq \widetilde{O}(1) \cdot \sup_{p} \sup_{\theta \in \mathbb{S}^{n-1}} \norm{\E_{x\sim p}\left(\langle x,\theta \rangle xx^T \right) }_F.
\end{align}
We replace the supremum over $\theta \in \mathbb{S}^{n-1}$ on the RHS by the expectation over $\mathbb{S}^{n-1}$. 
Here $\norm{\cdot}_F$ stands for the {\em Frobenius norm} (see Section~\ref{subsec:notationsdef}.
To see how~(\ref{eqn:BoundKLSbyTMB}) refines~(\ref{eqn:EldThinShellBound}), let $x,y \sim p$ be independent vectors and $\sigma$ be the uniform measure on $\mathbb{S}^{n-1}$. Then, 
\begin{align*}
\int_{\mathbb{S}^{n-1}} \norm{\E_{x\sim p}\left( \langle x,\theta \rangle xx^T  \right) }_F \dd \sigma(\theta)
&= \int_{\mathbb{S}^{n-1}} \E_{x,y \sim p} \left(\langle x,\theta \rangle \cdot \langle y,\theta \rangle \cdot \langle x,y\rangle^2 \right) \dd \sigma(\theta)\\
&= \frac{1}{n}\E_{x,y \sim p}\left( \langle x,y \rangle^3 \right).
\end{align*}

\paragraph{Acknowledgement.} We thank the anonymous referee for many helpful suggestions. 

\section{Preliminaries}
In this section, we review background definitions. 

\subsection{Notation and Definitions}
\label{subsec:notationsdef}
A function $h:\Rn \rightarrow \R_{+}$ is called {\em log-concave}
if it takes the form $h(x) = \exp(- f(x))$ for a convex function $f: \Rn \rightarrow \R \cup \{ \infty \}$.
It is {\em $t$-strongly log-concave} if it takes the form
$h(x)=h'(x)e^{-\frac{t}{2}\norm x_{2}^{2}}$
where $h'(x):\R^{n}\rightarrow \R_{+}$ is an integrable log-concave function.
A probability measure is log-concave ($t$-strongly log-concave) if it has a log-concave (resp. $t$-strongly log-concave) density function.

Given a matrix $A \in \R^{m \times n}$, we define its {\em Frobenius norm} (also known as Hilbert-Schmidt norm), denoted as $\norm{A}_F$, to be 
\[
\norm{A}_F = \sqrt{\sum_{i=1}^m \sum_{j=1}^n |A_{i,j}|^2} = \tr\left( A^T A \right).
\]
The {\em operator norm} (also known as spectral norm) of $A$, denoted  $\norm{A}_{\op}$, is defined as
\[
\norm{A}_{\op} = \sqrt{\lambda_{\max} \left(A^T A \right)},
\]
where $\lambda_{\max}(\cdot)$ stands for the maximum eigenvalue. 


\subsection{Stochastic calculus}
Given real-valued stochastic processes $x_{t}$
and $y_{t}$, the quadratic variations $[x]_{t}$ and $[x,y]_{t}$
are real-valued stochastic processes defined by
\[
[x]_{t}=\lim_{|P|\rightarrow0}\sum_{n=1}^{\infty}\left(x_{\tau_{n}}-x_{\tau_{n-1}}\right)^{2}\quad\text{and}\quad[x,y]_{t}=\lim_{|P|\rightarrow0}\sum_{n=1}^{\infty}\left(x_{\tau_{n}}-x_{\tau_{n-1}}\right)\left(y_{\tau_{n}}-y_{\tau_{n-1}}\right),
\]
where $P=\{0=\tau_{0}\leq\tau_{1}\leq\tau_{2}\leq\cdots\uparrow t\}$
is a stochastic partition of the non-negative real numbers, $|P|=\max_{n}\left(\tau_{n}-\tau_{n-1}\right)$
is called the \emph{mesh} of $P$ and the limit is defined using convergence
in probability. Note that $[x]_{t}$ is non-decreasing with $t$ and
$[x,y]_{t}$ can be defined as
\[
[x,y]_{t}=\frac{1}{4}\left([x+y]_{t}-[x-y]_{t}\right).
\]
For example, if the processes $x_{t}$ and $y_{t}$ satisfy the SDEs
$dx_{t}=\mu(x_{t})dt+\sigma(x_{t})dW_{t}$ and $dy_{t}=\nu(y_{t})dt+\eta(y_{t})dW_{t}$
where $W_{t}$ is a Wiener process, we have  $$[x]_{t}=\int_{0}^{t}\sigma^{2}(x_{s})ds
\quad [x,y]_{t}=\int_{0}^{t}\sigma(x_{s})\eta(y_{s})ds \mbox{ and } d[x,y]_{t}=\sigma(x_{t})\eta(y_{t})dt.$$
For vector-valued SDEs 
$$dx_{t}=\mu(x_{t})dt+\Sigma(x_{t})dW_{t} \mbox{ 
and } dy_{t}=\nu(y_{t})dt+M(y_{t})dW_{t},$$ 
we have that 
$$[x^{i},x^{j}]_{t}=\int_{0}^{t} \left(\Sigma(x_{s})\Sigma^{T}(x_{s}) \right)_{ij}ds
\mbox{ and } d[x^{i},y^{j}]_{t}=\left(\Sigma(x_{t})M^{T}(y_{t}) \right)_{ij}dt.$$
\begin{lem}[It\^{o}'s formula]\cite{ito1944}
\label{lem:Ito} Let $x$ be a semimartingale and $f$ be a twice
continuously differentiable function, then
\[
df(x_{t})=\sum_{i}\frac{df(x_{t})}{dx^{i}}dx^{i}+\frac{1}{2}\sum_{i,j}\frac{d^{2}f(x_{t})}{dx^{i}dx^{j}}d[x^{i},x^{j}]_{t}.
\]
\end{lem}

The next two lemmas are well-known facts about Wiener processes.
\begin{lem}[Reflection principle]
\label{lem:reflection}Given a Wiener process $W_t$ and $a,t\geq0$,
then we have that
\[
\P\left(\sup_{0\leq s\leq t}W_s\geq a \right)=2\P(W_t\geq a).
\]
\end{lem}

\begin{thm}
[Dambis, Dubins-Schwarz theorem]\cite{dambis1965,dubins1965}
\label{thm:Dubins}Every continuous local martingale $M_{t}$ is
of the form
\[
M_{t}=M_{0}+W_{[M]_{t}}\text{ for all }t\geq0, 
\]
where $W_{s}$ is a Wiener process. 
\end{thm}

\subsection{Log-concave functions}
\begin{thm}[Dinghas; Pr\'{e}kopa; Leindler]
\label{lem:marginal} The convolution of two log-concave functions
is log-concave; in particular, any
marginal of a log-concave density is log-concave.
\end{thm}
The next lemma is a ``reverse'' H\"{o}lder's inequality (see e.g., \cite{Lovasz2007}).
\begin{lem}[log-concave moments]
\label{lem:lcmom} For any log-concave density $p$ in $\R^{n}$
and any positive integer $k$, 
\[
\E_{x\sim p}\norm x^{k}\le(2k)^{k} \cdot \left(\E_{x\sim p}\norm x^{2}\right)^{k/2}.
\]
\end{lem}

The following inequality bounding the small ball probability is from
\cite{ArtsteinGM2015}.
\begin{thm}[{\cite[Thm. 10.4.7]{ArtsteinGM2015}}]
\label{lem:small-ball}For any isotropic log-concave density $p$
and any $\epsilon<\epsilon_{0}$, 
\[
\P_{x\sim p}\left(\norm {x}_2\le\epsilon\sqrt{n}\right)\le\epsilon^{c\sqrt{n}},
\]
where $\epsilon_{0},c$ are absolute constants. 
\end{thm}

The following theorem from~\cite{Mazja60,Cheeger69} states that the Poincar\'{e} constant is bounded by the KLS constant.
\begin{thm}[Poincar\'{e} Constant~\cite{Mazja60,Cheeger69}]
\label{thm:PoincareConst}
For any isotropic log-concave density $p$ in $\Rn$ and any smooth
function $g$, we have
\[
\Var_{x \sim p} g(x) \leq O\left(\psi_n^2 \right) \cdot \E_{x \sim p} \norm{\nabla g(x)}^2.
\]
\end{thm}

An immediate consequence of the above theorem is the following lemma which is central to our analysis. 
We give a proof of this central lemma for completeness.
\begin{lem}
\label{lem:quadratic-form} For any matrix $A$ and any isotropic
log-concave density $p$, 
\[
\Var_{x\sim p}\left(x^{T}Ax\right)\le O\left(\psi_{r}^{2} \right) \cdot \|A\|_F^2,
\]
where $r=\rank(A+A^{T})$.
\end{lem}
\begin{proof}
Since $x^T A x = x^T A^T x$, we have $\Var_{x \sim p} \left(x^T A x \right) = \Var_{x \sim p} \left(x^T \left(A + A^T \right) x \right)/4$.
Now applying Theorem~\ref{thm:PoincareConst} to the projection of $p$ onto the orthogonal complement of the null space of matrix $A$ finishes the proof.
\end{proof}



To prove a upper bound on the KLS constant, it suffices to consider subsets
of measure $1/2.$ We quote a theorem from \cite[Thm 1.8]{Milman2009}. 
\begin{thm}
\label{thm:milman}The KLS constant of any log-concave density
is achieved by a subset of measure $1/2.$
\end{thm}

The next theorem is an essentially best possible tail bound on large deviations for log-concave densities, due to Paouris \cite{Paouris2006}.

\begin{thm}\label{thm:Paouris}
There exists a universal constant $c$ such that for any isotropic log-concave density $p$ in $\R^n$ and any $t>1$, $\P_{x \sim p} \left(\|x\|>c\cdot t\sqrt{n} \right)\le e^{-t\sqrt{n}}.$
\end{thm}

\subsection{Distance between probability measures}
\label{subsec:DistProb}

The total variation distance is used in the statement of classical central limit theorem (e.g.~\cite{Klartag2007}). 
\begin{defn}
The total variation distance between two probability measures $\mu$ and $\nu$ in $\R$ is defined by
\[
\dtv(\mu,\nu) \overset{\Def}{=} \sup_{A \subseteq \R} \left|\mu(A)-\nu(A) \right|.
\]
\end{defn}

The following lemma relates total variation distance to $L_1$-Wasserstein distance (see Def. \ref{def:Wasser}) for isotropic log-concave distributions.
\begin{lem}[{\cite[Prop 1]{meckes2014equivalence}}]
\label{lem:TVtoWasser}
Let $\mu$ and $\nu$ be isotropic log-concave distributions in $\R$, then we have
\[
\dtv(\mu,\nu) = O(1) \cdot \sqrt{W_1(\mu,\nu)}.
\]
\end{lem}

Now we relate $L_s$ Wasserstein distance to $L_t$ Wasserstein distance for $1\leq s \neq t$.
By H\"{o}lder's inequality, one can show that for any $s \leq t$, we have $W_s(\mu,\nu) \leq W_t(\mu,\nu)$.
In the special case where both $\mu$ and $\nu$ are isotropic log-concave distributions in $\R$, it is shown in~\cite[Prop 5]{meckes2014equivalence} that
\[
W_t(\mu,\nu)^t \leq  O(1)  \cdot  W_s(\mu,\nu)^s \log^{t-s}\left( \frac{t^t}{W_s(\mu,\nu)^s} \right).
\]
In the following, we generalize this result to cases where $\mu$ or $\nu$ might be the measure of the inner product of two independent isotropic log-concave vectors. 
This generalization might be useful for future applications.
The proof is essentially the same as that in \cite{meckes2014equivalence} as is therefore postponed to Appendix~\ref{sec:missProofsSecDistProb}.

\begin{restatable}{lem}{WqboundbyWp}
\label{lem:WqboundbyWp}
Let $\mu$ and $\nu$ be two probability measures in $\R$. Suppose one of the following holds:
\begin{enumerate}
    \item Both $\mu$ and $\nu$ are isotropic log-concave distributions.
    \item The distribution $\mu$ is isotropic log-concave, while $\nu$ is the measure of the random variable $\frac{1}{\sqrt{n}}\langle x,y \rangle$ where $x \sim p$ and $y \sim q$ are independent random vectors and $p,q$ are isotropic log-concave distributions in $\Rn$.
    \item There exist isotropic log-concave distributions $p_\mu,q_\mu,p_\nu$ and $q_\nu$ in $\Rn$ such that $\mu$ is the measure of the random variable $\frac{1}{\sqrt{n}}\langle x_\mu,y_\mu \rangle$ and $\nu$ is the measure of the random variable 
    $\frac{1}{\sqrt{n}}\langle x_\nu,y_\nu \rangle$, where $x_\mu \sim p_\mu$, $y_\mu \sim q_\mu$, $x_\nu \sim p_\nu$ and $y_\nu \sim q_\nu$ are independent random vectors.
\end{enumerate}
Then there exists a universal constant $c > 0$ such that for any $1 \leq s < t$, we have
\begin{align*}
W_t(\mu,\nu)^t \leq  c W_s(\mu,\nu)^s \log^{t-s}\left( \frac{c^t t^{2t}}{W_s(\mu,\nu)^s} \right) + c^t t^{2t} \exp(-c\sqrt{n}).
\end{align*}
Moreover, the above bound is valid even when the coupling $(\mu,\nu)$ on the left-hand side is taken to be the best coupling for $W_s(\mu,\nu)$ instead of the best coupling for $W_t(\mu,\nu)$.
\end{restatable}


\subsection{Matrix inequalities}

For any symmetric matrix $B$, we define $|B|=\sqrt{B^{2}}$, namely,
the matrix formed by taking absolute value of all eigenvalues of $B$.
\begin{lem}[Matrix H\"{o}lder inequality]
\label{lem:matrixholder}Given a symmetric matrices $A$ and $B$
and any $s,t\geq1$ with $s^{-1}+t^{-1}=1$, we have
\[
\tr(AB)\leq\left(\tr\left|A\right|^{s}\right)^{1/s}\left(\tr\left|B\right|^{t}\right)^{1/t}.
\]
\end{lem}

\begin{lem}[Lieb-Thirring Inequality \cite{lieb1976inequalities}]
\label{lem:lieborg}Given positive semi-definite matrices $A$ and
$B$ and $r\geq1$, we have
\[
\tr \left( \left(B^{1/2}AB^{1/2} \right)^{r} \right)\leq\tr \left(B^{r/2}A^{r}B^{r/2} \right).
\]
\end{lem}
\begin{lem}[\cite{Eldan2013,allen2016using}]
\label{lem:lieb}Given a symmetric matrix $B$, a positive semi-definite
matrix $A$ and $\alpha\in[0,1]$, we have
\[
\tr\left(A^{\alpha}BA^{1-\alpha}B \right)\leq\tr\left(AB^{2} \right).
\]
\end{lem}

\subsection{From Generalized CLT to Third Moment Bound}
\label{subsec:GenCLTtoTMB}

In this subsection, we prove that an improved bound for Generalized CLT implies an improved third moment bound.
\begin{thm}
\label{thm:GenCLTtoTMB}
Fix $\epsilon \in (0,1/2)$. Let $p$ be any isotropic log-concave distribution in $\Rn$, $x,y$ be independent random vectors drawn from $p$ and $G\sim \Gauss(0,n)$.
If we have
\begin{eqnarray}
\label{eqn:ExpgeneralizedCLT}
W_2(\langle x,y \rangle, G)^2 = O\left(n^{1-2\epsilon} \right),
\end{eqnarray}
then it follows that
$$
\E_{x,y \sim p} \left(\langle x,y \rangle^3 \right) = O \left(n^{1.5-\epsilon} \right).
$$
\end{thm}

We remark that while the equivalence between Generalized CLT and the KLS conjecture in our main theorem (Theorem~\ref{thm:GCLTEquivKLS}) does not hold in a point-wise sense, the result in Theorem~\ref{thm:GenCLTtoTMB} holds for every isotropic log-concave $p$.

\begin{proof}
Let $\pi_2$ be the best coupling between $\langle x,y \rangle$ and $G$ in~(\ref{eqn:ExpgeneralizedCLT}). In the rest of the proof, we use $\E_{\pi_2}$ to denote the expectation where $\langle x,y \rangle$ and $G$ satisfies the coupling $\pi_2$.
Applying Lemma~\ref{lem:WqboundbyWp}, we have
\[
\E_{\pi_2} |\langle x,y \rangle, G|^3 = O \left(n^{\frac{3}{2}-2\epsilon} \log n \right).
\]

Now we can bound $\E_{x,y \sim p} \langle x,y \rangle^3$  using the coupling $\pi_2$ as 
\[
 \E_{x,y \sim p} \langle x,y \rangle^3  = \E_{\pi_2} \left(\langle x,y \rangle - G + G \right)^3 =  \E_{\pi_2} \left( G^3 + 3 G^2 (\langle x,y \rangle-G)+ 3G(\langle x,y \rangle-G)^2 + (\langle x,y \rangle-G)^3 \right).
\]
The first term is zero due to symmetry.
For the second term, we have
\begin{align*}
\E_{\pi_2} G^2 (\langle x,y \rangle-G)
&\leq \sqrt{\E_{G \sim N(0,n)} G^4} \cdot  \sqrt{\E_{\pi_2} (\langle x,y \rangle-G)^2} \\
& = O(n) \cdot O\left(n^{0.5-\epsilon} \right)  = O\left(n^{1.5-\epsilon} \right).
\end{align*}

The last two terms can be bounded similarly as
\begin{align*}
\E_{\pi_2} G (\langle x,y \rangle-G)^2 
&\leq \left(\E_{G \sim N(0,n)} |G|^3 \right)^{\frac{1}{3}} \cdot \left(\E_{\pi_2} |\langle x,y \rangle-G|^3 \right)^{\frac{2}{3}} \\
& = O\left(\sqrt{n} \right) \cdot O\left(n^{1-\frac{4}{3}\epsilon} \log^{\frac{2}{3}} n \right) = O\left(n^{1.5-\epsilon} \right),
\end{align*}
and
\begin{eqnarray*}
\E_{\pi_2} (\langle x,y \rangle-G)^3
\leq \E_{\pi_2} |\langle x,y \rangle-G|^3
= O\left(n^{1.5-2\epsilon} \log n \right) = O\left(n^{1.5 - \epsilon} \right).
\end{eqnarray*}
This completes the proof of Theorem~\ref{thm:GenCLTtoTMB}.
\end{proof}

\section{Stochastic Localization}
\label{sec:StochasticLocalization}

The key technique used in part of our proofs is the stochastic localization scheme introduced in~\cite{Eldan2013}. 
The idea is to transform a given log-concave density into one that is proportional to a Gaussian times the original density. 
This is achieved by a martingale process by modifying the current density infinitesimally according to an exponential in a random direction. By having a martingale, the measures of subsets are maintained in expectation, and the challenge is to control how close they remain to their expectations over time.
We now define a simple version of the process we will use, which is the same as in \cite{lee2016arxiv}.

\subsection{The process and its basic properties}
Given a distribution with a log-concave density $p(x)$, we start at
time $t=0$ with this distribution and at each time $t>0$, we apply
an infinitesimal change to the density. This is done by picking a
random direction from a standard Gaussian.

\begin{defn}
\label{def:A}Given a log-concave distribution $p$, we define the
following stochastic differential equation:
\begin{equation}
c_{0}=0,\quad dc_{t}=dW_{t}+\mu_{t}dt,\label{eq:dBt}
\end{equation}
where the probability distribution $p_{t}$, the mean $\mu_{t}$ and
the covariance $A_{t}$ are defined by
\[
p_{t}(x)=\frac{e^{c_{t}^{T}x-\frac{t}{2}\norm x_{2}^{2}}p(x)}{\int_{\Rn}e^{c_{t}^{T}y-\frac{t}{2}\norm y_{2}^{2}}p(y)dy},\quad\mu_{t}=\E_{x\sim p_{t}}x,\quad A_{t}=\E_{x\sim p_{t}}(x-\mu_{t})(x-\mu_{t})^{T}.
\]
\end{defn}

\comment{
Since $\mu_{t}$ is a bounded function that is Lipschitz with respect
to $c$ and $t$, standard theorems (e.g. \cite[Sec 5.2]{oksendal2013stochastic})
show the existence and uniqueness of the solution in time $[0,T]$
for any $T>0$.
Now we proceed to analyzing the process and how its parameters evolve.
Roughly speaking, the first lemma below says that the stochastic process
is the same as continuously multiplying $p_{t}(x)$ by a random infinitesimally
small linear function. 
}
The following basic lemmas will be used in the analysis. For a more rigorous account of the construction and further details of the process, the reader is referred to \cite{eldan2018clt,lee2016arxiv,LVSurvey18}

\begin{lem}
\label{lem:def-pt}
For any $x\in\Rn$, we have  $dp_{t}(x)=(x-\mu_{t})^{T}dW_{t}p_{t}(x)$.
\end{lem}

\comment{
By considering the derivative $d\log p_{t}(x)$, we see that applying
$dp_{t}(x)$ as in the lemma above results in the distribution $p_{t}(x)$,
with the Gaussian term in the density:
\begin{align*}
d\log p_{t}(x) & =\frac{dp_{t}(x)}{p_{t}(x)}-\frac{1}{2}\frac{d[p_{t}(x)]_{t}}{p_{t}(x)^{2}}=(x-\mu_{t})^{T}dW_{t}-\frac{1}{2}(x-\mu_{t})^{T}(x-\mu_{t})dt\\
 & =x^{T}dc_{t}-\frac{1}{2}\norm x^{2}dt+g(t)
\end{align*}
where the last term is independent of $x$ and the first two terms
explain the form of $p_{t}(x)$ and the appearance of the Gaussian. 
}

Next we state the change of the mean and the covariance matrix.
\begin{lem}
\label{lem:dA} $d\mu_t = A_t d W_t$ and  $dA_{t}=\int_{\Rn}(x-\mu_{t})(x-\mu_{t})^{T}\left((x-\mu_{t})^{T}dW_{t}\right)p_{t}(x)dx-A_{t}^{2}dt.$
\end{lem}

\subsection{Bounding the KLS constant}
The following lemmas from ~\cite{lee2016arxiv} are used to bound the KLS constant by the spectral norm of the covariance matrix at time $t.$ First, we bound the measure of a set of initial
measure $\frac{1}{2}$. 
\begin{lem}
\label{lem:volumeKLS}For any set $E\subset\R^n$ with $\int_{E}p(x)dx=\frac{1}{2}$
and $t\geq0$, we have that
\[
\P\left(\frac{1}{4}\leq\int_{E}p_{t}(x)dx\leq\frac{3}{4}\right)\geq\frac{9}{10}-\P\left(\int_{0}^{t}\norm{A_{s}}_{\spe}ds\geq\frac{1}{64}\right).
\]
\end{lem}

\comment{
\begin{proof}
Let $g_{t}=\int_{E}p_{t}(x)dx$. Then, we have that
\[
dg_{t}=\left\langle \int_{E}(x-\mu_{t})p_{t}(x)dx,dW_{t}\right\rangle 
\]
where the integral might not be $0$ because it is over the subset
$E$ and not all of $\R^{n}$. Hence, we have
\begin{align*}
d[g_{t}]_{t} & =\norm{\int_{E}(x-\mu_{t})p_{t}(x)dx}_{2}^{2}dt=\max_{\norm{\zeta}_{2}\leq1}\left(\int_{E}\zeta^{T}(x-\mu_{t})p_{t}(x)dx\right)^{2}dt\\
 & \leq\left(\max_{\norm{\zeta}_{2}\leq1}\int_{\Rn}\left(\zeta^{T}(x-\mu_{t})\right)^{2}p_{t}(x)dx\right)\left(\int_{\Rn}p_{t}(x)dx\right)dt\\
 & =\max_{\norm{\zeta}_{2}\leq1}\left(\zeta^{T}A_{t}\zeta\right)dt=\norm{A_{t}}_{\spe}dt.
\end{align*}
Hence, we have that $\frac{d[g_{t}]_{t}}{dt}\leq\norm{A_{t}}_{\spe}$.
By the Dambis, Dubins-Schwarz theorem, there exists a Wiener process
$\tilde{W}_{t}$ such that $g_{t}-g_{0}$ has the same distribution
as $\tilde{W}_{[g]_{t}}$. Using $g_{0}=\frac{1}{2}$, we have that
\begin{align*}
\P(\frac{1}{4}\leq g_{t}\leq\frac{3}{4}) & =\P(\frac{-1}{4}\leq\tilde{W}_{[g]_{t}}\leq\frac{1}{4})\geq1-\P(\max_{0\leq s\leq\frac{1}{64}}\left|\tilde{W}_{s}\right|>\frac{1}{4})-\P([g]_{t}>\frac{1}{64})\\
 & \overset{\cirt1}{\geq}1-4\P(\tilde{W}_{\frac{1}{64}}>\frac{1}{4})-\P([g]_{t}>\frac{1}{64})\\
 & \overset{\cirt2}{\geq}\frac{9}{10}-\P([g]_{t}>\frac{1}{64})
\end{align*}
where we used reflection principle for 1-dimensional Brownian motion
in $\cirt1$ and the concentration of normal distribution in $\cirt2,$
namely $\P_{x\sim N(0,1)}$$\left(x>2\right)\le0.0228$. 
\end{proof}
}

At time $t$, the distribution is $t$-strongly log-concave and it
is known that it has KLS constant $O\left(t^{-1/2}\right)$. The following
isoperimetric inequality was proved in \cite{CV2014} and was also
used in \cite{Eldan2013}. 
\begin{thm}
\label{thm:Gaussian-iso}Let $h(x)=f(x)e^{-\frac{t}{2}\norm x_{2}^{2}}/\int f(y)e^{-\frac{t}{2}\norm y_{2}^{2}}dy$
where $f:\R^{n}\rightarrow\R_{+}$ is an integrable log-concave function.
Then $h$ is log-concave and for any measurable subset $S$ of $\R^{n}$,
\[
\int_{\partial S}h(x)dx=\Omega\left(\sqrt{t}\right) \cdot \min\left\{ \int_{S}h(x)dx,\int_{\R^{n}\setminus S}h(x)dx\right\} .
\]
In other words, the KLS constant of h is $O\left(t^{-1/2}\right)$.
\end{thm}
This gives a bound on the KLS constant. 
\begin{lem}
\label{lem:boundAgivesKLS}Given a log-concave distribution $p$, let
$A_{t}$ be given by Definition \ref{def:A} using initial distribution
$p$. Suppose that there is $T>0$ such that
\[
\P\left(\int_{0}^{T}\norm{A_{s}}_{\op}ds\leq\frac{1}{64}\right)\geq\frac{3}{4},
\]
then we have $\psi_{p}=O\left(T^{-1/2}\right)$.
\end{lem}

\comment{
\begin{proof}
By Milman's theorem \cite{Milman2009}, it suffices to consider subsets
of measure $\frac{1}{2}.$ Consider any measurable subset $E$ of
$\R^{n}$ of initial measure $\frac{1}{2}$. By Lemma \ref{lem:def-pt},
$p_{t}$ is a martingale and therefore
\[
\int_{\partial E}p(x)dx=\int_{\partial E}p_{0}(x)dx=\E\left(\int_{\partial E}p_{t}(x)dx\right).
\]
Next, by the definition of $p_{T}$ (\ref{eq:dBt}), we have that
$p_{T}(x)\propto e^{c_{T}^{T}x-\frac{T}{2}\norm x^{2}}p(x)$ and Theorem
\ref{thm:Gaussian-iso} shows that the expansion of $E$ is $\Omega\left(\sqrt{T}\right)$.
Hence, we have
\begin{align*}
\int_{\partial E}p(x)dx & =\E\int_{\partial E}p_{T}(x)dx\gtrsim\sqrt{T}\cdot\E\left(\min\left(\int_{E}p_{T}(x)dx,\int_{\bar{E}}p_{T}(x)dx\right)\right)\\
 & \gtrsim\sqrt{T}\cdot\P\left(\frac{1}{4}\leq\int_{E}p_{T}(x)dx\leq\frac{3}{4}\right)\overset{\lref{volumeKLS}}{\gtrsim}\sqrt{T}\cdot\left(\frac{9}{10}-\P(\int_{0}^{t}\norm{A_{s}}_{\op}ds\geq\frac{1}{64})\right)=\Omega(\sqrt{T})
\end{align*}
where we used the assumption at the end. Using Theorem \ref{thm:milman},
this shows that $\psi_{p}=O\left(T^{-1/2}\right).$
\end{proof}
}

Thus to prove a bound on $\psi_p$, it suffices to give an upper bound on $\norm{A_t}_{\op}$.
The potential function we will use to bound $\norm{A_t}_{\op}$ is $\Phi_t = \tr((A_t-I)^q)$ for some even integer $q$.
We give the detailed analysis in Section~\ref{sec:TMBtoKLS}.

The following result  from~\cite{lee2016arxiv} will be useful. It shows that the operator norm stays bounded up to a certain time with probability close to $1$.
\begin{lem}[\cite{lee2016arxiv}, Lemma 58]
\label{lem:BoundOperNorm}
Assume for $k \geq 1$, $\psi_p = O(n^{1/2k})$ for any isotropic log-concave distribution $p$ in $\Rn$.
There is a constant $c \geq 0$ s.t. for any 
$$0 \leq T \leq \frac{1}{c \cdot k \cdot (\log n)^{1-\frac{1}{k}} \cdot n^{1/k}},$$ 
we have
\begin{align}
\label{eqn:BoundOperNorm}
\p \left[\max_{t \in [0,T]} \norm{A_t}_\op \geq 2 \right] \leq 2 \exp\left( -\frac{1}{cT} \right).
\end{align}
\end{lem}

\subsection{Bounding the potential}
In order to bound the potential $\Phi_t = \tr((A_t-I)^q)$, we bound its derivative. We go from the derivative to the potential itself via the following lemma, which might also be useful in future applications.
\begin{lem}
\label{lem:BoundPotential}
Let $\{\Phi_t\}_{t \geq 0}$ be an $n$-dimensional It\^{o} process with $\Phi_0 \leq \frac{U}{2}$ and $d \Phi_t = \delta_t \dd t + v_t^T \dd W_t$. Let $T>0$ be some fixed time, $U>0$ be some target upper bound, and $f$ and $g$ be some auxiliary functions such that for all $0 \leq t \leq T$
\begin{enumerate}
    \item $\delta_t \leq f(\Phi_t)$ and $\norm{v_t}_2 \leq g(\Phi_t)$,
    \item Both $f(\cdot)$ and $g(\cdot)$ are non-negative non-decreasing functions,
    \item $f(U) \cdot T \leq \frac{U}{8}$ and $g(U) \cdot \sqrt{T} \leq \frac{U}{8}$.
\end{enumerate}
Then, we have the following upper bound on $\Phi_t$:
\[
\P\left[ \max_{t \in [0,T]} \Phi_t \geq U\right] \leq 0.01.
\]
\end{lem}

\begin{proof}
We denote the It\^{o} process formed by the martingale term as $\{Y_t\}_{t \geq 0}$, i.e. $Y_0 = 0$ and $d Y_t = v_t^T \dd W_t$.
We first show that in order to control $\Phi_t$, it suffices to control $Y_t$.
\begin{claim}
\label{claim:BoundPhibyY}
For any $0\leq t_0 \leq T$, if $\max_{t \in [0,t_0]} Y_t \leq \frac{U}{3}$,
then we have
\[
\max_{t \in [0,t_0]} \Phi_t \leq U.
\]
\end{claim}
\begin{proofof}{Claim~\ref{claim:BoundPhibyY}}
Assume for the purpose of contradiction that $\max_{t \in [0,t_0]} \Phi_t > U$.
Denote $t' = \inf \{t \in [0,t_0]|\Phi_t \geq U\} $.
It follows that for any $t \in [0,t']$, we have $\Phi_t \leq U$ and $f(\Phi_t) \cdot t' \leq f(U) \cdot T \leq \frac{U}{8}$.
It follows that
\[
\Phi_{t} \leq \Phi_0 + \frac{U}{8} + Y_{t} < U,
\]
which leads to a contradiction.
\end{proofof}
Since $Y_t$ is a martingale, it follows from Theorem~\ref{thm:Dubins} that there exists a Wiener process $\{B_t\}_{t \geq 0}$ such that $Y_t = B_{[Y]_t}$, for all $t \geq 0$.
The next claim bounds $Y_t$ using $B_t$.
\begin{claim}
\label{claim:BoundYbyB}
If $\max_{t \in \left[0,U^2/64 \right]} B_t \leq \frac{U}{3}$,
then we have
\[
\max_{t \in [0,T]} Y_t \leq U/3,
\]
\end{claim}
\begin{proofof}{Claim~\ref{claim:BoundYbyB}}
Assume for the purpose of contradiction that $\max_{t \in [0,T]} Y_t \geq \frac{U}{3}$. Define $t_0$ as the first time when $Y_t$ becomes at least $\frac{U}{3}$.
By definition, for any $t\in [0,t_0]$, $Y_t \leq \frac{U}{3}$.
Using Claim~\ref{claim:BoundPhibyY}, we have 
$\max_{t \in [0,t_0]} \Phi_t \leq U$.
It follows that 
\[
[Y]_{t_0} = \int_{0}^{t_0} \norm{v_t}_2^2 \dd t \leq T \cdot g^2(U) \leq \frac{U^2}{64}.
\]
This implies that
\[
Y_{t_0} = B_{[Y]_{t_0}} \leq \max_{t \in \left[0,U^2/64 \right]} B_t \leq \frac{U}{3},
\]
which leads to a contradiction.
\end{proofof}
Now it suffices to bound the probability that the Wiener process $\{B_t\}_{t \geq 0}$ exceeds $U/3$ in the time period $[0,U^2/64]$.
Using the reflection principle in Lemma~\ref{lem:reflection}, we have
\begin{align*}
\Pr\left[ \max_{t \in [0,T]} \Phi_t \geq U\right] \leq \Pr\left[ \max_{t \in \left[0,U^2/64 \right]} B_t > U/3 \right] = 2 \Pr\left[  B_{U^2/64} > U/3 \right] \leq 0.01.
\end{align*}
\end{proof}

\comment{
\section{From Generalized CLT to KLS}

In this section we show that a generalized version of central limit theorem implies an improved bound on KLS constant.
We prove the following theorem.
\begin{thm}
\label{thm:GenCLTtoKLS}
Fix $\eps \in (0,1/2)$. If for any isotropic log-concave measure $p$ in $\Rn$ and independent vectors $x,y\sim p$ and $g\sim \Gauss(0,n)$, we have \begin{eqnarray}
\label{eqn:GeneralizedCLT}
W_2(\langle x,y \rangle, G)^2 = O(n^{1-2\epsilon}),
\end{eqnarray}
then for any $\delta > 0$, we have $\psi_n = O(n^{1/4 - \eps/2 + \delta})$.
\end{thm}

The theorem is implied by Theorem~\ref{thm:GenCLTtoTMB} and Theorem~\ref{thm:TMBtoKLS}, which we will prove in Section~\ref{subsec:GenCLTtoTMB} and Section~\ref{subsec:TMBtoKLS} respectively.

}

\eat{
\begin{thm}
\label{thm:GenCLTtoTMB}
Let $p$ be any isotropic log-concave distribution with dimension $n$, $x,y$ be independent random vectors drawn from $p$ and $G\sim \Gauss(0,n)$.
If  
\begin{eqnarray}
\label{eqn:ExpgeneralizedCLT}
\E_{y} \left(W_4(\langle x,y \rangle, G)\right)^4 = O(n^{2-4\epsilon}),
\end{eqnarray}
then 
$$
\E_{x,y \sim p} (\langle x,y \rangle)^3 = O(n^{1.5-\epsilon}).
$$
\end{thm}

\begin{proof}
We abbreviate $\langle x,y \rangle$ as $\langle x,y \rangle$. 
For a chosen $y$, denote the coupling in Equation~\ref{eqn:ExpgeneralizedCLT} as $(\langle x,y \rangle, G)$.
Then 
\begin{eqnarray*}
&& \E_{x,y \sim p} \langle x,y \rangle^3  = \E_{\pi_2} (\langle x,y \rangle - G + G)^3 \\
&&=  \E_{\pi_2} \left( G^3 + 3 G^2 (\langle x,y \rangle-G)+ 3G(\langle x,y \rangle-G)^2 + (\langle x,y \rangle-G)^3 \right)
\end{eqnarray*}
The first term is 0 due to symmetry.
For the second term
\begin{eqnarray*}
&&\E_{\pi_2} G^2 (\langle x,y \rangle-G) \\
&&\leq \sqrt{\E_{G \sim N(0,n)} G^4} \sqrt{\E_{\pi_2} (\langle x,y \rangle-G)^2} \\
&& = O(1) \cdot n \cdot O(n^{0.5-\epsilon})\\
&& = O(n^{1.5-\epsilon})
\end{eqnarray*}

The last two terms are bounded similarly
\begin{eqnarray*}
&&\E_{\pi_2} G (\langle x,y \rangle-G)^2\\
&&\leq \sqrt{\E_{G \sim N(0,n)} G^2} \sqrt{\E_{\pi_2} (\langle x,y \rangle-G)^4} \\
&& = O(1) \sqrt{n} O(n^{1-2\epsilon})\\
&& = O(n^{1.5-2\epsilon})
\end{eqnarray*}
and
\begin{eqnarray*}
&&\E_{\pi_2} (\langle x,y \rangle-G)^3\\
&& \leq \sqrt{\E_{\pi_2} (\langle x,y \rangle-G)^4} \sqrt{\E_{\pi_2} (\langle x,y \rangle-G)^2}\\
&& = O(n^{1.5-3\epsilon})
\end{eqnarray*}
Summing everything up finishes the proof of the lemma.
\end{proof}
}
\section{From Third Moment Bound to KLS}
\label{sec:TMBtoKLS}


In this section, we show that an improved third moment bound implies an improved bound on the KLS constant.  Theorems ~\ref{thm:TMBtoKLS} and \ref{thm:GenCLTtoTMB} together imply the first part of Theorem~\ref{thm:GCLTEquivKLS}.

\TMBtoKLS*

The rest of this section is devoted to proving Theorem~\ref{thm:TMBtoKLS}. 
Throughout this section, we assume the condition in Theorem~\ref{thm:TMBtoKLS} holds, i.e. for every isotropic log-concave distribution $p$ in $\Rn$ and independent vectors $x,y \sim p$, one has
\begin{eqnarray}
\E_{x,y \sim p} \left(\langle x,y\rangle^3 \right) = O\left(n^{1.5-\epsilon} \right).
\end{eqnarray}

\subsection{Tensor inequalities}
\label{subsubsec:TensorBounds}
The proof of Theorem~\ref{thm:TMBtoKLS} is based on the potential function $\Phi_t = \tr \left((A_t-I)^q \right)$ for some even integer $q$.
This potential is the one of the key technical differences between this paper and previous work using stochastic localization, which used $\tr(A_t^q)$ \cite{Eldan2013,lee2017KLS}. The proof of a tight log-Sobolev inequality \cite{lee2018stochastic} used a Stieltjes-type potential function, $\tr((uI-A)^{-q})$ to avoid logarithmic factors. The potential we use here, $\tr \left((A_t-I)^q \right)$ allows us to track how close $A_t$ is to $I$ (not just bounding how large $A_t$ is). For example, in Lemma \ref{lem:dPhi_p}, we bound the derivative of the potential $\Phi_t$ by some powers of $\Phi_t$. Since $\Phi_t$ is $0$ initially, this gives a significantly tighter bound around $t=0$ (compared to $\tr(A_t^q)$). We will discuss this again in the course of the proof.

For the analysis we define the following tensor and derive some of its properties.

\begin{defn}[3-Tensor]
For an isotropic log-concave distribution $p$ in $\Rn$ and symmetric matrices $A,B$ and $C$, define 
\[
T_p(A,B,C) = \E_{x,y\sim p}\left(x^T A y \right) \left(x^T B y\right)\left(x^T C y \right)
\]
We drop the subscript $p$ to indicate the worst case bound over all isotropic log-concave distributions
\[
T(A,B,C) \overset{\Def}{=} \sup_{\textrm{isotropic log-concave } p} \E_{x,y\sim p}\left(x^T A y \right) \left(x^T B y\right)\left(x^T C y \right)
\]
\end{defn}
It is clear from the definition that $T$ is invariant under permutation of $A,B$ and $C$. 
In the rest of this subsection, we give a few tensor inequalities that will be used throughout the rest of our proofs. 
The proofs of these tensor inequalities are postponed to Appendix~\ref{sec:missProofsTensorBounds}.

\begin{restatable}{lem}{trabs}
\label{lem:trabs} For any $A_{1},A_{2},A_{3}\succeq0$, we have that $T(A_{1},A_{2},A_{3})\geq0$
and for any symmetric matrices $B_{1},B_{2},B_{3}$, we have that
\[
T(B_{1},B_{2},B_{3})\leq T\left(\left|B_{1}\right|,\left|B_{2}\right|,\left|B_{3}\right| \right).
\]
\end{restatable}

In the next lemma, we collect tensor inequalities that will
be useful for later proofs.
\begin{restatable}{lem}{tinq}
\label{lem:tinq}Suppose that $\psi_{k}\leq\alpha k^{\beta}$ for
all $k\leq n$ for some fixed $0\leq\beta\leq\frac{1}{2}$ and $\alpha\geq1$.
For any isotropic log-concave distribution $p$ in $\Rn$ and symmetric
matrices $A$ and $B$, we have that
\begin{enumerate}
\item $T(A,I,I)\leq T(I,I,I)  \cdot \norm {A}_\op$.\label{lem:tinq1}
\item $T(A,I,I)\leq O\left(\psi_{n}^{2}\right) \cdot \tr\left|A\right|$.\label{lem:tinq2}
\item $T(A,B,I)\leq O\left(\psi_{r}^{2} \right) \cdot \norm {B}_{\op}\tr\left|A\right|$ where
$r=\min(2\cdot \rank(B),n)$. \label{lem:tinq3}
\item $T(A,B,I)\leq O\left(\alpha^{2}\log n \right) \cdot \left(\tr\left|B\right|^{1/(2\beta)}\right)^{2\beta}\tr\left|A\right|$.\label{lem:tinq4}
\item $T(A,B,I)\leq\left(T\left(\left|A\right|^{s},I,I \right)\right)^{1/s} \cdot \left(T \left(\left|B\right|^{t},I,I \right)\right)^{1/t}$,
for any $s,t\geq1$ with $s^{-1}+t^{-1}=1$.\label{lem:tinq5}
\end{enumerate}
\end{restatable}

\begin{restatable}{lem}{liebtr}
\label{lem:liebtr}For any positive semi-definite matrices $A,B,C$
and any $\alpha\in[0,1]$, then 
\[
T\left(B^{1/2}A^{\alpha}B^{1/2},B^{1/2}A^{1-\alpha}B^{1/2},C \right)\leq T \left(B^{1/2}AB^{1/2},B,C \right).
\]
\end{restatable}

\subsection{Derivatives of the potential}

The next lemma computes the derivative of $\Phi_t = \tr((A_t-I)^q)$, as done in \cite{lee2016arxiv}. For the reader's convenience, we include a proof here.
\begin{lem}
\label{lem:trace_generalized}Let $A_{t}$ be defined by Definition
\ref{def:A}. For any integer $q\geq2$, we have that
\begin{align*}
d\tr\left((A_{t}- I)^{q} \right)= & q \cdot \E_{x\sim p_{t}}(x-\mu_{t})^{T}(A_{t}- I)^{q-1}(x-\mu_{t})(x-\mu_{t})^{T}dW_{t}-q \cdot \tr \left((A_{t}- I \right)^{q-1}A_{t}^{2})dt\\
 & +\frac{q}{2} \cdot \sum_{\alpha+\beta=q-2}\E_{x,y\sim p_{t}}(x-\mu_{t})^{T}(A_{t}- I)^{\alpha}(y-\mu_{t})(x-\mu_{t})^{T}(A_{t}- I)^{\beta}(y-\mu_{t})(x-\mu_{t})^{T}(y-\mu_{t})dt.
\end{align*}
\end{lem}

\begin{proof}
Let $\Phi(X)=\tr((X- I)^{q})$. Then the first and second-order
directional derivatives of $\Phi$ at $X$ is given by 
\[
\left.\frac{\partial\Phi}{\partial X}\right|_{H}=q \cdot \tr \left((X- I)^{q-1}H\right)\quad\text{and}\quad\left.\frac{\partial^{2}\Phi}{\partial X\partial X}\right|_{H_{1},H_{2}}=q \cdot \sum_{k=0}^{q-2}\tr\left((X- I)^{k}H_{2}(X- I)^{q-2-k}H_{1} \right).
\]
Using these and It\^{o}'s formula, we have that
\[
d\tr((A_{t}- I)^{q})=q \cdot \tr \left((A_{t}- I)^{q-1}dA_{t} \right)+\frac{q}{2} \cdot \sum_{\alpha+\beta=q-2}\sum_{ijkl}\tr \left((A_{t}- I)^{\alpha}e_{ij}(A_{t}- I)^{\beta}e_{kl} \right)d[A_{ij},A_{kl}]_{t},
\]
where $e_{ij}$ is the matrix that is $1$ in the entry $(i,j)$ and
$0$ otherwise, and $A_{ij}$ is the real-valued stochastic process
defined by the $(i,j)^{th}$ entry of $A_{t}$.

Using Lemma \ref{lem:dA} and Lemma \ref{lem:def-pt}, we have that
\begin{align}
dA_{t} & =\E_{x\sim p_{t}}(x-\mu_{t})(x-\mu_{t})^{T}(x-\mu_{t})^{T}dW_{t}-A_{t}A_{t}dt\nonumber \\
 & =\E_{x\sim p_{t}}(x-\mu_{t})(x-\mu_{t})^{T}(x-\mu_{t})^{T}e_{z}dW_{t,z}-A_{t}A_{t}dt,\label{eq:dAz2_gen}
\end{align}
where $W_{t,z}$ is the $z^{th}$ coordinate of $W_{t}$. Therefore,
\begin{align}
d[A_{ij},A_{kl}]_{t} & =\sum_{z}\left(\E_{x\sim p_{t}}(x-\mu_{t})_{i}(x-\mu_{t})_{j}(x-\mu_{t})^{T}e_{z}\right)\left(\E_{x\sim p_{t}}(x-\mu_{t})_{k}(x-\mu_{t})_{l}(x-\mu_{t})^{T}e_{z}\right)dt\nonumber \\
 & =\E_{x,y\sim p_{t}}(x-\mu_{t})_{i}(x-\mu_{t})_{j}(y-\mu_{t})_{k}(y-\mu_{t})_{l}(x-\mu_{t})^{T}(y-\mu_{t})dt.\label{eq:dAv2_gen}
\end{align}

Using the formula for $dA_{t}$ (\ref{eq:dAz2_gen}) and $d[A_{ij},A_{kl}]_{t}$
(\ref{eq:dAv2_gen}), we have that
\begin{align*}
 & d\tr \left((A_{t}- I)^{q} \right)\\
= & q \cdot \E_{x\sim p_{t}}(x-\mu_{t})^{T}(A_{t}- I)^{q-1}(x-\mu_{t})(x-\mu_{t})^{T}dW_{t}-q \cdot \tr\left((A_{t}- I)^{q-1}A_{t}^{2} \right)dt\\
 & +\frac{q}{2} \cdot \sum_{\alpha+\beta=q-2}\sum_{ijkl}\tr\left((A_{t}- I)^{\alpha}e_{ij}(A_{t}- I)^{\beta}e_{kl} \right)\E_{x,y\sim p_{t}}(x-\mu_{t})_{i}(x-\mu_{t})_{j}(y-\mu_{t})_{k}(y-\mu_{t})_{l}(x-\mu_{t})^{T}(y-\mu_{t})dt\\
= & q \cdot \E_{x\sim p_{t}}(x-\mu_{t})^{T}(A_{t}- I)^{q-1}(x-\mu_{t})(x-\mu_{t})^{T}dW_{t}-q \cdot \tr\left((A_{t}- I)^{q-1}A_{t}^{2} \right)dt\\
 & +\frac{q}{2} \cdot \sum_{\alpha+\beta=q-2}\E_{x,y\sim p_{t}}(x-\mu_{t})^{T}(A_{t}- I)^{\alpha}(y-\mu_{t})(x-\mu_{t})^{T}(A_{t}- I)^{\beta}(y-\mu_{t})(x-\mu_{t})^{T}(y-\mu_{t})dt.
\end{align*}
\end{proof}

\subsection{Bounding the Potential}
The derivative of the potential has drift ($dt$) and stochastic/Martingale ($dW_t$) terms.
The next lemma bounds the drift and Martingale parts of the change in the potential by tensor quantities. We will then bound each one separately.

\begin{lem}
\label{lem:dPhi_p}Let $A_{t}$ and $p_{t}$ be defined as in Definition
\ref{def:A}. Let $\Phi_{t}=\tr((A_{t}-I)^{q})$ for some even integer
$q\geq2$, then we have that $d\Phi_{t}=\delta_{t}dt+v_{t}^{T}dW_{t}$
with
\[
\delta_{t}\leq\frac{1}{2}q(q-1) \cdot T\left(A_{t}(A_{t}-I)^{q-2},A_{t},A_{t} \right)+2q \cdot \left(\Phi_{t}^{1+\frac{1}{q}}+\Phi_{t}^{1-\frac{1}{q}}n^{\frac{1}{q}} \right)
\]
and
\[
\norm{v_{t}}_{2}\leq q \cdot \norm{\E_{x\sim p}(x-\mu_{t})^{T}(A-I)^{q-1}(x-\mu_{t})(x-\mu_{t})^{T}}_{2}.
\]
\end{lem}
\begin{proof}
By Lemma~\ref{lem:trace_generalized}, we have 
\begin{align*}
d\Phi_{t}= & q\cdot \E_{x\sim p_{t}}(x-\mu_{t})^{T}(A_{t}-I)^{q-1}(x-\mu_{t})(x-\mu_{t})^{T}dW_{t}-q \cdot \tr\left((A_{t}-I)^{q-1}A_{t}^{2} \right)dt\\
 & +\frac{q}{2} \cdot \sum_{\alpha+\beta=q-2}\E_{x,y\sim p_{t}}(x-\mu_{t})^{T}(A_{t}-I)^{\alpha}(y-\mu_{t})(x-\mu_{t})^{T}(A_{t}-I)^{\beta}(y-\mu_{t})(x-\mu_{t})^{T}(y-\mu_{t})dt\\
= & q \cdot \E_{x\sim p}(x-\mu_{t})^{T}(A-I)^{q-1}(x-\mu_{t})(x-\mu_{t})^{T}dW_{t}
-q \cdot \tr\left((A_{t}-I)^{q-1}A_{t}^{2} \right)dt\\
 & +\frac{q}{2} \cdot \sum_{\alpha+\beta=q-2}\E_{x,y\sim\tilde{p}_{t}}x^{T}A_{t}(A_{t}-I)^{\alpha}yx^{T}A_{t}(A_{t}-I)^{\beta}yx^{T}A_{t}ydt\\
\defeq & \delta_{t}dt+v_{t}^{T}dW_{t}.
\end{align*}
where $\tilde{p}_{t}$ is the isotropic correspondance of $p_{t}$ defined
by $\tilde{p}_{t}(x)=p \left(A_{t}^{1/2}x+\mu_{t} \right)$, $\delta_{t}dt$ is
the drift term in $d\Phi_{t}$ and $v_{t}^{T}dW_{t}$ is the martingale
term in $d\Phi_{t}$.

For the drift term $\alpha_{t}dt$, we have 
\[
\delta_{t}\leq\frac{q}{2} \cdot \sum_{\alpha+\beta=q-2}T\left(A_{t}(A_{t}-I)^{\alpha},A_{t}(A_{t}-I)^{\beta},A_{t} \right)-q \cdot \tr\left((A_{t}-I)^{q-1}A_{t}^{2}\right).
\]
The first drift term is
\begin{align*}
\frac{q}{2} \cdot \sum_{\alpha+\beta=q-2}T\left(A_{t}(A_{t}-I)^{\alpha},A_{t}(A_{t}-I)^{\beta},A_{t} \right)\leq & \frac{q}{2}\cdot \sum_{\alpha+\beta=q-2} T\left(A_{t}\left|A_{t}-I\right|^{\alpha},A_{t}\left|A_{t}-I\right|^{\beta},A_{t} \right)\ttag{\lref{trabs}}\\
\leq & \frac{q}{2}\cdot \sum_{\alpha+\beta=q-2}T\left(A_{t}\left|A_{t}-I\right|^{q-2},A_{t},A_{t} \right)\ttag{\lref{liebtr}}\\
= & \frac{q(q-1)}{2} \cdot T\left(A_{t}(A_{t}-I)^{q-2},A_{t},A_{t} \right).
\end{align*}
For the second drift term, since $q$ is even, we have that
\begin{align*}
-q \cdot \tr\left((A_{t}-I)^{q-1}A_{t}^{2} \right) \leq & q \cdot \tr\left(|A_{t}-I|^{q-1}(A_{t}-I+I)^{2} \right)\\
\leq & 2q \cdot \tr\left(|A_{t}-I|^{q+1} \right)+2q \cdot \tr\left(|A_{t}-I|^{q-1} \right)\\
\leq & 2q \cdot \Phi_{t}^{1+\frac{1}{q}}+2q \cdot \Phi_{t}^{1-\frac{1}{q}}n^{\frac{1}{q}}.
\end{align*}

For the Martingale term $v_{t}^{T}dW_{t}$, we note that
\begin{align*}
\norm{v_{t}}_{2} & =q \cdot \norm{\E_{x\sim p}(x-\mu_{t})^{T}(A-I)^{q-1}(x-\mu_{t})(x-\mu_{t})^{T}}_2.
\end{align*}
\end{proof}

The Martingale term is relatively straightforward to bound. We use the following lemma from~\cite{lee2016arxiv} in our analysis.
\begin{lem}[{\cite[Lem 25]{lee2016arxiv}}] 
\label{lem:tensorestimate}
Given a log-concave distribution $p$ with mean
$\mu$ and covariance $A$. For any positive semi-definite matrix
$C$, we have that
\[
\norm{ \E_{x \sim p} (x-\mu)(x-\mu)^T C (x-\mu)}_2 = O\left( \norm{A}_\spe^{1/2} \cdot \tr \left(A^{1/2}CA^{1/2} \right)\right).
\]
\end{lem}

\begin{lem}
\label{lem:MartingaleTerm}
Let $p_t$ be the log-concave distribution at time $t$ with covariance matrix $A_t$. Let $\Phi_t = \tr((A_t - I)^q)$ for some even integer $q \geq 2$ and $\dd \Phi_t = \delta_t \dd t + v_t^T \dd W_t$. Assume $\Phi_t \leq n$. Then,
\[
||v_t||_2 \leq q \cdot \norm{\E_{x \sim {p_t}}(x- \mu_t)^T (A-I)^{q-1} (x-\mu_t) (x-\mu_t)^T }_2 
\leq O\left( q \right) \cdot  \left( \Phi_t^{1-\frac{1}{2q}} n^{\frac{1}{q}}+ n^{\frac{1}{q}} \right).
\]
\end{lem}

\begin{proof}
Note that
\begin{align*}
\norm{\E_{x\sim p}(x-\mu_{t})^{T}(A_t-I)^{q-1}(x-\mu_{t})(x-\mu_{t})^{T}}_{2} & \leq O\left(1\right) \cdot \norm {A_t}_\spe^{1/2}\tr\left|A_t^{1/2}(A_t-I)^{q-1}A_t^{1/2}\right|\ttag{\lref{tensorestimate}}\\
 & \leq O\left(1\right) \cdot \norm {A_t}_\spe ^{1/2}\tr|A_t-I|^{q-1}+O\left(1\right) \cdot \norm {A_t}_\spe^{1/2}\tr|A_t-I|^{q}\\
 & \leq O\left(1+\Phi_t^{\frac{1}{2q}} \right) \cdot \Phi_t^{1-\frac{1}{q}}n^{\frac{1}{q}}+O\left(1+\Phi_t^{\frac{1}{2q}} \right) \cdot \Phi_t\\
 & \leq O\left(\Phi_t^{1-\frac{1}{2q}}n^{\frac{1}{q}}+\Phi_t^{1+\frac{1}{2q}}+n^{\frac{1}{q}} \right).
\end{align*}
\end{proof}

Next we bound the drift term. This takes more work.  
We write 
\[
\delta_t \leq \frac{1}{2}q(q-1) \delta^{(1)}_t + q \delta^{(2)}_t, 
\]
where 
\[
\delta^{(1)}_t = T\left(A_t(A_t-I)^{q-2}, A_t, A_t \right)
\quad\text{and}\quad
\delta^{(2)}_t = \Phi_t^{1+\frac{1}{q}}+ \Phi_t^{1-\frac{1}{q}}n^{\frac{1}{q}}.
\]

We bound $\delta^{(1)}_t$ in the following lemma. This is the core lemma which needs several tensor properties and bounds. It is also the reason we use $\tr((A_t-I)^q$ as the potential. Specifically, using this potential lets us write $A-I$ as the sum of two matrices one with small eigenvalues and the other of low rank, by choosing the threshold for ``small" eigenvalue appropriately. 

\begin{lem}
\label{lem:DriftTerm}
Suppose that $\psi_k \leq \alpha k^\beta$ for all $k \leq n$ for some $\alpha \ge 1$ and $\beta$ s.t. $1/4-\epsilon/2 \leq \beta \leq 1/4$.
Let $\Phi = \tr((A-I)^q)$ for some even integer $q \geq \frac{1}{2\beta}$ and $\Lambda = 4\beta + 2\epsilon - 1$. Assume $\Phi \leq n$. Then
\begin{align*}
\label{eqn:DriftTerm}
\delta^{(1)} \leq O(\alpha^2) \cdot \Phi n^{2\beta} \cdot
\left[ n^{-\frac{1}{q}}\Phi^{\frac{1}{q}} \log n
+ n^{-\frac{\Lambda}{4q}} \cdot n^{\frac{2}{q}} \Phi^{-\frac{2}{q}}
\right].
\end{align*}
\end{lem}
\begin{proof}
We have that 
\begin{align*}
\delta^{(1)} &= T \left(A(A-I)^{q-2}, A, A \right) \\
& = T\left((A-I)^{q-1}+(A-I)^{q-2}, A-I + I, A-I + I \right)\\
& \leq T\left(|A-I|^{q-1}, |A-I|, |A-I| \right) + 2T\left(|A-I|^{q-1}, |A-I|, I \right)+ T\left(|A-I|^{q-1}, I, I \right) \ttag{\lref{trabs}} \\
&\quad + T\left((A-I)^{q-2}, |A-I|, |A-I| \right) + 2T \left((A-I)^{q-2}, |A-I|, I \right) + T\left((A-I)^{q-2}, I, I \right)\\
&\leq T\left(|A-I|^{q-1}, |A-I|, |A-I| \right) + 3T\left(|A-I|^{q-1}, |A-I|, I \right) \\
&\quad + 3T\left(|A-I|^{q-1}, I, I \right) + T\left((A-I)^{q-2}, I, I \right)\ttag{\lref{liebtr}}\\
& \overset{\Delta}{=}\delta^{(1)}_1 + 3\delta^{(1)}_2 + 3\delta^{(1)}_3 + \delta^{(1)}_4.
\end{align*}

We first bound $\delta^{(1)}_1$ as follows
\begin{align*}
\delta^{(1)}_1 &= T \left(|A-I|^{q-1}, |A-I|, |A-I| \right) \\
&\leq T\left(|A-I|^q, |A-I|, I \right)\ttag{\lref{liebtr}}\\
& \leq O(\alpha^{2}\log n) \cdot \Phi\left(\tr|A-I|^{1/2\beta}\right)^{2\beta}\ttag{\lreff{tinq}4}\\
&\leq O(\alpha^2 \log n)\cdot \Phi \left( \left(\tr|A-I|^q \right)^\frac{1}{2\beta q} n^{1-\frac{1}{2\beta q}} \right)^{2 \beta}\ttag{\lref{matrixholder}}\\
& \leq O(\alpha^2 \log n) \cdot n^{2\beta - \frac{1}{q}}\Phi^{1+\frac{1}{q}}.
\end{align*}

For $\delta^{(1)}_2$, we write 
\[
|A-I| = B_1 + B_2,
\]
where $B_1$ consists of the eigen-components of $|A-I|$ with eigenvalues at most $\eta$ and $B_2$ is the remaining part. Then we can bound $\delta^{(1)}_2$ as follows
\begin{align}
\label{eqn:DriftSecondTerm}
\delta^{(1)}_2 = T \left(B_1^{q-1}, B_1, I \right) + T \left(B_1^{q-1}, B_2, I \right) + T \left(B_2^{q-1}, B_1, I \right) + T \left(B_2^{q-1}, B_2, I \right).
\end{align}
The first term in Equation (\ref{eqn:DriftSecondTerm}) can be bounded as
\begin{align*}
T \left(B_1^{q-1}, B_1, I \right) &\leq T \left(B_1^{q}, I, I \right)\ttag{\lref{liebtr}} \\
&\leq T\left(I,I,I \right) \cdot ||B_1||^q\ttag{\lreff{tinq}1}\\
&\leq O\left(\eta^q n^{1.5-\epsilon} \right).
\end{align*}

The second term in Equation (\ref{eqn:DriftSecondTerm}) is bounded as
\begin{align*}
T \left(B_1^{q-1}, B_2, I \right) &\leq T \left(B_1^{q}, I, I \right)^{\frac{q-1}{q}} \cdot T \left(B_2^{q}, I, I \right)^{\frac{1}{q}} \ttag{\lreff{tinq}5}\\
&\leq O\left(\eta^q n^{1.5-\epsilon} \right)^{\frac{q-1}{q}} \cdot O\left(\psi_n^2 \Phi \right)^{\frac{1}{q}}\ttag{\lreff{tinq}1\text{ and }\lreff{tinq}2}\\
& = O(1) \cdot \alpha^{\frac{2}{q}} \eta^{q-1} n^{ \frac{(1.5-\epsilon)(q-1)}{q} 
+ \frac{2\beta}{q}} \Phi^{\frac{1}{q}},
\end{align*}
where we used $\tr\left(B_2^q \right) \leq \tr \left((A-I)^q \right) \leq \Phi$ in the last line. For the third term in Equation (\ref{eqn:DriftSecondTerm}), we have 
\begin{align*}
T \left(B_2^{q-1}, B_1, I \right) & \leq T\left(B_2^q,I,I \right)^{\frac{q-1}{q}}\cdot T\left(B_1^q,I,I \right)^{\frac{1}{q}}\ttag{\lreff{tinq}5}\\
& \leq O\left(\psi_n^2 \Phi \right)^{\frac{q-1}{q}} \cdot O\left(\eta^q n^{1.5-\epsilon} \right)^{\frac{1}{q}}\ttag{\lreff{tinq}1\text{ and }\lreff{tinq}2}\\
& =O(1) \cdot \alpha^{\frac{2(q-1)}{q}} \eta n^{2\beta - \frac{2\beta}{q} 
+ (1.5 - \epsilon) \cdot \frac{1}{q}} \Phi^{\frac{q-1}{q}}.
\end{align*}

For the last term in Equation (\ref{eqn:DriftSecondTerm})
, let $P$ be the orthogonal projection from $\R^n$ to the range of $B_2$. Notice that $\rank(B_2) \leq \frac{\Phi}{\eta^q}$ because each positive eigenvalue of $B_2$ is at least $\eta$. We have
\begin{align*}
T \left(B_2^{q-1}, B_2, I \right) & = T\left(PB_2^{q-1}P,  PB_2P, I \right)\\
&\leq T\left(P B_2^q P, P, I \right)\ttag{\lref{liebtr}}\\
&\leq O\left(\psi_{2 \cdot \rank(B_2)}^2 \right) \cdot \Phi\ttag{\lreff{tinq}3}\\
&= O \left( \frac{\alpha^2 \Phi^{1+2\beta}}{\eta^{2\beta q}} \right).
\end{align*}

Summing up these four terms, we get
\begin{align*}
\delta^{(1)}_2 
&\leq O(1) \cdot \left[ \eta^q n^{1.5-\epsilon} 
+ \alpha^{\frac{2}{q}} \eta^{q-1} n^{ \frac{(1.5-\epsilon)(q-1)}{q} 
+ \frac{2\beta}{q}} \Phi^{\frac{1}{q}} 
+ \alpha^{\frac{2(q-1)}{q}} \eta n^{2\beta - \frac{2\beta}{q} 
+ (1.5 - \epsilon) \cdot \frac{1}{q}} \Phi^{\frac{q-1}{q}} 
+ \frac{\alpha^2 \Phi^{1+2\beta}}{\eta^{2\beta q}}\right] \\
&\leq O(\alpha^2) \cdot \left[ \eta^q n^{1.5-\epsilon} 
+ \eta^{q-1} n^{ \frac{2\beta + (1.5-\epsilon)(q-1)}{q}} \Phi^{\frac{1}{q}} 
+ \eta n^{2\beta - \frac{2\beta}{q} 
+ (1.5 - \epsilon) \cdot \frac{1}{q}} \Phi^{\frac{q-1}{q}} + \frac{ \Phi^{1+2\beta}}{\eta^{2\beta q}} \right].
\end{align*}

It turns out that when $1/4 - \epsilon/2 \leq \beta \leq 1/4$,
the last two terms dominate the first two terms (which is justified shortly).
Balancing the last two terms, we choose $\eta = \Phi^{\frac{1}{q}} n^{-\frac{2\beta(q-1) + 1.5-\epsilon}{q(1+2 \beta q)}}$, and this gives
\begin{align*}
\delta^{(1)}_2 \leq O(\alpha^2) \cdot \left[ \Phi n^{2\beta}\cdot n^{\frac{\beta(1-4\beta -2\epsilon)q}{1+2\beta q}} 
+ \Phi n^{2\beta}\cdot n^{\frac{\beta(1-4\beta -2\epsilon)(q-1)}{1+2\beta q}} 
+  \Phi n^{2\beta}\cdot n^{\frac{\beta(1-4\beta -2\epsilon)}{1+2\beta q}} 
+\Phi n^{2\beta}\cdot n^{\frac{\beta(1-4\beta -2\epsilon)}{1+2\beta q}}
\right].
\end{align*}
Since $\beta \geq 1/4-\epsilon/2$, $\beta(1-4\beta - 2\epsilon) \leq 0$ which implies that the last two terms dominate the first two terms in this case.
We therefore have
\[
\delta^{(1)}_2 \leq O(\alpha^2) \cdot \Phi n^{2\beta}\cdot n^{\frac{\beta(1-4\beta -2\epsilon)}{1+2\beta q}}.
\]

The third term $\delta^{(1)}_3$ is bounded as 
\begin{align*}
\delta^{(1)}_3 &= T\left(|A-I|^{q-1}, I, I \right)\\
& = T\left(B_1^{q-1},I,I \right) + T \left(B_2^{q-1},I,I \right)\\
&\leq O(1) \cdot \left(\eta^{q-1} n^{1.5-\epsilon} + \alpha^2 n^{2\beta} \Phi /\eta\right)\ttag{\lreff{tinq}1\text{ and }\lreff{tinq}2}\\
& \leq O(\alpha^2) \cdot n^{\frac{2\beta(q-1) + 1.5 - \epsilon}{q}} \Phi^{\frac{q-1}{q}},
\end{align*}
where the last line is by choosing $\eta = \left( n^{2\beta - 1.5+\epsilon} \Phi \right)^{1/q}$. The final term $\delta^{(1)}_4$ is bounded as
\begin{align*}
\delta^{(1)}_4 &= T\left(|A-I|^{q-2}, I, I \right)\\
&=T\left(B_1^{q-2},I,I \right) + T \left(B_2^{q-2},I,I \right)\\
& \leq O(1) \cdot \left( \eta^{q-2} n^{1.5- \epsilon} + \alpha^2 n^{2\beta} \Phi/\eta^2 \right) \ttag{\lreff{tinq}1\text{ and }\lreff{tinq}2}\\
& \leq O(\alpha^2) \cdot  n^{\frac{2\beta(q-2) + 2(1.5-\epsilon)}{q}} \Phi^{\frac{q-2}{q}}.
\end{align*}    

Combining all the terms we have
\[
\delta^{(1)} \leq O(\alpha^2) \cdot \Phi n^{2\beta} \cdot \left[ n^{-\frac{1}{q}} \Phi^{\frac{1}{q}} \log n + n^{-\frac{\beta}{1+2 \beta q} \cdot \Lambda} + n^{-\frac{\Lambda}{2q}} n^{\frac{1}{q}} \Phi^{-\frac{1}{q}} + n^{-\frac{\Lambda}{q}} n^{\frac{2}{q}} \Phi^{-\frac{2}{q}} \right].
\]
Simplifying the above with the assumptions $\Phi \leq n$ and $q \geq \frac{1}{2\beta}$ finishes the proof of the lemma.
\end{proof}                

\subsection{Proof of Theorem~\ref{thm:TMBtoKLS}}
We note that 
$\Phi_0 = 0$.
Using the bounds we have, 
we will show that when $q$ is taken as the smallest even integer greater than $\max\{8,\lceil 1/\delta \rceil\}$,
with probability close to 1, we can write
\[
\Phi_t \leq O\left(n^{1-\frac{\Lambda}{12}} \log^{-q} n \right),
\]
for all $t \in [0, T]$ where $T = O\left( \frac{n^{-2\beta + \frac{\Lambda}{24q}}}{\alpha^2} \right)$.

Intuitively, when $\Phi_t \leq O\left(n^{1-\frac{\Lambda}{12}} \log^{-q} n \right)$ and $T = O\left( \frac{n^{-2\beta + \frac{\Lambda}{24q}}}{\alpha^2} \right)$, we have, using the analysis of the previous section,
\[
\delta_t T \leq O\left(n^{1-\frac{\Lambda}{12}} \log^{-q} n \right)
\quad \text{and} \quad
\norm{v_t}_2 \sqrt{T} \leq O\left(n^{1-\frac{\Lambda}{12}} \log^{-q} n \right).
\]
This suggests that $\Phi_t$ stays at most $O\left(n^{1-\frac{\Lambda}{12}} \log^{-q} n \right)$ during a period of length $T$.
Formally, we prove the following lemma to get an improved bound on $\psi_n$.
Our proof applies Lemma~\ref{lem:BoundPotential}.

\begin{lem}
\label{lem:ImprovedStoppingTime}
Suppose that $\psi_k \leq \alpha k^\beta,\forall k\leq n$ for some $\alpha \geq 1$ and $1/4-\epsilon/2 < \beta \leq 1/4$. Let $p$ be any isotropic log-concave distribution. Let $\Phi_t = \tr((A_t - I)^q)$ with $q = 2 \lceil 1/\beta \rceil$. Then for $n$ large enough such that  $n^{\frac{\Lambda}{48q}} > \log n$ where $\Lambda = 4\beta + 2\epsilon - 1$, there exists a universal constant $C$ s.t.
\[
\p \left[\max_{t \in [0,T]} \Phi_t \geq n^{1-\frac{\Lambda}{12}} \log^{-q} n \right] \leq 0.01 \quad \text{with} \quad T= \frac{C n^{-2\beta + \frac{\Lambda}{24q}}}{\alpha^2}.
\]
\end{lem}
\begin{proof}
We use Lemma~\ref{lem:BoundPotential} with the bounds from Lemma~\ref{lem:MartingaleTerm} and~\ref{lem:DriftTerm}.
Recall we have the following bound on the potential change. 
\[
\dd \Phi_t = \delta_t \dd t + v_t^T \dd W_t,
\]
with $||v_t||_2  
\leq g(\Phi_t)$
where $g(\Phi_t)$ is defined to be $+\infty$ when $\Phi_t > n$ and $O\left( q \right) \cdot \left( \Phi_t^{1-\frac{1}{2q}} n^{\frac{1}{q}} + n^{\frac{1}{q}} \right)$ otherwise,
and $\delta_t \leq f(\Phi_t)$ where $f(\Phi_t)$ is defined to be $+\infty$ when $\Phi_t > n$ and $\frac{1}{2}q(q-1) \delta^{(1)}(\Phi_t) + q \delta^{(2)}(\Phi_t)$ otherwise
where 
\[
\delta^{(1)}(\Phi_t) = O(\alpha^2) \cdot  \Phi_t n^{2\beta} \cdot 
\left[  n^{-\frac{1}{q}}\Phi_t^{\frac{1}{q}} \log n
+ n^{-\frac{\Lambda}{4q}} \cdot n^{\frac{2}{q}} \Phi_t^{-\frac{2}{q}}
\right],
\]
and
\[
\delta^{(2)}(\Phi_t) = \Phi_t^{1+\frac{1}{q}}+ \Phi_t^{1-\frac{1}{q}}n^{\frac{1}{q}}.
\]
We show that the conditions in Lemma~\ref{lem:BoundPotential} are met with $U = n^{1-\frac{\Lambda}{12}} \log^{-q} n$ and $T = \frac{C n^{-2\beta + \frac{\Lambda}{24q}}}{\alpha^2}$ for some small enough constant $C$.
It is easy to see that $f(\Phi_t)$ and $g(\Phi_t)$ are non-negative and non-decreasing functions of $\Phi_t$ by our choice of $q$, so we only need to check that the last condition of Lemma \ref{lem:BoundPotential} holds.

We first consider the martingale term. For $1 \leq U \leq n$, we have
\begin{align*}
g(U) \cdot \sqrt{T} &= O\left( q \right) \cdot\left( U^{1-\frac{1}{2q}} n^{\frac{1}{q}} + n^{\frac{1}{q}} \right) \cdot \frac{\sqrt{C} n^{-\beta + \frac{\Lambda}{48q}}}{\alpha^2} \\
&\leq O(q) \cdot U \cdot U^{-\frac{1}{2q}} n^{\frac{1}{q}} \cdot \frac{\sqrt{C} n^{-\beta + \frac{\Lambda}{48q}}}{\alpha^2} \\
&\leq U \cdot O(q) \cdot \sqrt{C} \cdot n^{-\beta + \frac{1}{q} + \frac{\Lambda}{48q}}.
\end{align*}
Note that $q \geq 2/\beta$ and $\Lambda \leq 1$.
Thus, 
\[
g(U) \cdot \sqrt{T} \leq U \cdot O(q)\sqrt{C}.
\]
which is bounded by $U/8$ when $C$ is small enough.

Now we verify that $f(U) \cdot T \leq U/8$ for some suitably small constant $C$.
We first verify this for $\delta^{(2)}(\Phi_t)$.
\begin{align*}
\delta^{(2)}(U) \cdot T & \leq U \cdot \left(U^{\frac{1}{q}} + U^{-\frac{1}{q}}n^{\frac{1}{q}} \right) C n^{-2\beta + \frac{\Lambda}{24q}}\\
&= U \cdot C \left( n^{\frac{1}{q} - \frac{\Lambda}{12q}} \log^{-1} n + n^{\frac{\Lambda}{12q}} \log n \right) n^{-2\beta + \frac{\Lambda}{24q}}\\
&\leq U C n^{-2\beta + \frac{1}{q} - \frac{\Lambda}{24q}} \log n \\
&\leq U C,
\end{align*}
where in the last line we used $q \geq 2/\beta$, $\Lambda \leq 1$ and $n^{\beta} > \log n$. Now we consider $\delta^{(1)}(\Phi_t)$. 
We denote the two terms in $\delta^{(1)}(\Phi_t)$ as $\delta^{(1)}_i(\Phi_t)$, where $i=1,2$. For the first term $\delta^{(1)}_1(\Phi_t)$ we have
\begin{align*}
\delta^{(1)}_1(U) \cdot T &= O(\alpha^2) \cdot U n^{2\beta}
 (\log n) n^{-\frac{1}{q}} U^{\frac{1}{q}} \cdot \frac{C n^{-2\beta + \frac{\Lambda}{24q}}}{\alpha^2} \\
 &= O(1) \cdot  U C n^{-\frac{\Lambda}{24q}} \\ &\leq O(1) \cdot  U C.
\end{align*}
For the second term $\delta^{(1)}_2(\Phi_t)$ we have
\begin{align*}
\delta^{(1)}_2(U) \cdot T &= O(\alpha^2) \cdot U n^{2\beta} \cdot  n^{-\frac{\Lambda}{4q}} \cdot n^{\frac{2}{q}} U^{-\frac{2}{q}} \cdot \frac{C n^{-2\beta + \frac{\Lambda}{24q}}}{\alpha^2} \\
& = O(1) \cdot U C n^{-\frac{\Lambda}{24q}} \log^2 n \\
&\leq O(1) \cdot U C.
\end{align*}
This shows that 
\[
\delta^{(1)}(U) T \leq O(1) U C.
\]
Thus, for some suitably small $C$, we have $f(U) \cdot T \leq U/8$. 
Applying Lemma~\ref{lem:BoundPotential} completes the proof of the lemma.
\end{proof}

When $1/4-\epsilon/2 < \beta \leq 1/4$, we get a better bound on $\psi_n$.
\begin{lem}
\label{lem:ImprovedKLSBound}
Suppose that $\psi_k \leq \alpha k^\beta$, for all $k\leq n$ for some $\alpha \geq 1$ and $1/4-\epsilon/2 < \beta \leq 1/4$. Let $p$ be an isotropic log-concave distribution in $\R^n$. Then for $n$ large enough such that $n^{\frac{\Lambda}{48q}} > \log n$, there exists a universal constant $C >0$ s.t.
\[
\psi_n \leq C \alpha n^{\beta - \frac{\Lambda}{48q}},
\]
where $\Lambda = 4\beta + 2\epsilon-1$ and $q = 2\lceil 1/\beta \rceil$.
\end{lem}
\begin{proof}
Using Lemma~\ref{lem:ImprovedStoppingTime}, with probability at least 0.99, for any $t \leq T = \frac{C n^{-2\beta + \frac{\Lambda}{24q}}}{\alpha^2}$ where $C$ is some universal constant and $q=2 \lceil 1/\beta \rceil$, we have
\[
\Phi_t \leq n^{1-\frac{\Lambda}{12}} \log^{-q} n.
\]
Assuming this event, we have
\[
\int_0^T ||A_t||_\op \dd t \leq \int_0^T \left(1+\Phi_t^{1/q} \right) \leq T\left( 1+n^{\frac{1}{q} - \frac{\Lambda}{12q}} \log^{-1}n \right) \leq 1/64.
\]
Now applying Lemma~\ref{lem:boundAgivesKLS}, we get
\[
\psi_p \leq O(\alpha) \cdot n^{\beta - \frac{\Lambda}{48q}},
\]
where $C$ is some universal constant.
Since $p$ is arbitrary, we have the result.
\end{proof}

Now we are finally ready to prove Theorem~\ref{thm:TMBtoKLS}. 

\begin{proofof}{Theorem~\ref{thm:TMBtoKLS}}
We start with the known bound $\psi_n \leq \alpha_0 n^{\beta_0}$ for $\beta_0 = 1/4$ and some constant $\alpha_0$. 
We construct a sequence of better and better bounds for $\psi_n$ which hold for any $n$ large enough such that $n^{\frac{\Lambda}{48q}} > \log n$, where $q=\Theta(1/\beta)=O(1/(1-2\epsilon+4\delta))$. 
(Note that if $\Lambda \leq 4\delta$, then we are done by Lemma \ref{lem:ImprovedKLSBound}. So we can assume without loss of generality that $\Lambda > 4\delta$). Since $q$ is fixed, one can find a fixed $n_0$ such that for any $n \geq n_0$, the requirement $n^{\frac{\Lambda}{48q}} > \log n$ is satisfied whenever $\Lambda > 4\delta$, regardless of the current bound on $\psi_n$. 

Suppose $\psi_n \leq \alpha_i n^{\beta_i}$ is the current bound. If $\beta_i \leq 1/4-\epsilon/2+\delta$, then we are done. Otherwise, applying Lemma~\ref{lem:ImprovedKLSBound} gives the better bound
$$
\psi_n \leq \alpha_{i+1} n^{\beta_{i+1}},
$$
where $\alpha_{i+1} = C \alpha_i$ and $\beta_{i+1} = \beta_i - \frac{\Lambda}{48q} \leq \beta_i - \frac{\delta}{12q}$ (since $\Lambda \ge 4\delta$).
Therefore, starting from $\beta_0 = 1/4$ and repeating the procedure at most $M =\lceil \frac{6\epsilon q}{\delta}\rceil$ times, we will get some $m \leq M$ such that $\psi_n \leq \alpha_m n^{\beta_m}$ where $\beta_m \leq 1/4-\epsilon/2+\delta$ and $\alpha_m \leq C ^{\lceil \frac{3q}{\delta}\rceil} \alpha_0$.
This holds for any large $n$ such that $n^{\frac{\delta}{12q}} > \log n$.
For small $n$ that doesn't satisfy the requirement $n^{\frac{\delta}{12q}} > \log n$, we simply bound them by some constant. 
We conclude that $\psi_n \leq O\left(n^{1/4-\epsilon/2+\delta} \right)$ for any $n$. We note that in fact the bound we get is 
$n^{1/4 - \epsilon/2 + \delta + q/(\delta \log n)}$ and since $q = O(1/\beta)$, we can set $\delta = O(1/\sqrt{\beta \log n})$ so that the bound on $\beta$ is $1/4 - \eps/2 + o(1)$.  
\end{proofof}

\section{From KLS to Generalized CLT}
\begin{thm}
\label{thm:KLStoGenCLT}
Assume $\psi_n = O(n^{1/4-\epsilon/2})$ for some $0 < \epsilon < 1/2$ and some dimension $n$.
Let $p,q$ be any isotropic log-concave distributions in $\Rn$, $x,y$ be independent random vectors drawn from $p$ and $q$ and $G\sim \Gauss(0,n)$.
It follows that
\begin{align}
\label{eqn:ExpgeneralizedCLTW2}
W_2 (\langle x,y \rangle, G )^2 = O\left(n^{1-2\epsilon}(\log n)^{1/2 + \epsilon} \right).
\end{align}
\end{thm}

This gives exactly the condition in Theorem~\ref{thm:GenCLTtoTMB} (up to a small polynomial factor in $\log n$).
The remainder of this section is devoted to proving Theorem~\ref{thm:KLStoGenCLT}. We start by relating $\langle x,y \rangle$ with $\langle x,g \rangle$, where $x\sim p$, $y \sim q$ are independent vectors drawn from isotropic log-concave distributions $p,q$ in $\R^n$ and $g\sim \Gauss(0,I)$ is a standard Gaussian vector in $\R^n$.

\begin{lem}
\label{lem:InnerProdlog-concaveGauss}
Assume the conditions of Theorem~\ref{thm:KLStoGenCLT}. Let $g\sim\Gauss(0,I)$ be independent from $x$ and $y$, then we have
\begin{align*}
W_2 (\langle x,y \rangle, \langle x,g \rangle )^2 = O \left(n^{1-2\epsilon} (\log n)^{1/2 + \epsilon} \right).
\end{align*}
\end{lem}

Before we prove Lemma~\ref{lem:InnerProdlog-concaveGauss}, we show how to use the lemma to prove Theorem~\ref{thm:KLStoGenCLT}.
The intuition is the following.
Lemma~\ref{lem:InnerProdlog-concaveGauss} allows us to relate $\langle x,y \rangle$ to $\langle x,g \rangle$.
Notice for fixed $x$, the random variable $\langle x,g \rangle$ has a Gaussian law with variance $\norm{x}_2$. 
Since $\norm{x}_2$ is concentrated around $\sqrt{n}$, it follows that $\langle x,g \rangle$ is close to the Gaussian distribution $\Gauss(0,n)$.

\begin{proofof}{Theorem~\ref{thm:KLStoGenCLT} Using Lemma~\ref{lem:InnerProdlog-concaveGauss}}
Let $g$ be a random vector drawn from a standard $n$-dimensional normal distribution $\Gauss(0,I)$.
By Lemma~\ref{lem:InnerProdlog-concaveGauss}, we have
\begin{align}
\label{eqn:ytoGauss}
W_2 (\langle x,y \rangle, \langle x,g \rangle )^2 = O \left(n^{1-2\epsilon}(\log n)^{1/2 + \epsilon} \right).
\end{align}
For fixed sample $x$, the random variable $\langle x,g \rangle$ has the same law as $\norm{x}_2 \cdot g_1$ where $g_1 \sim \Gauss(0,1)$.
Notice that $G$ has the same law as $\sqrt{n} \cdot g_2$, where $g_2 \sim \Gauss(0,1)$.
When $x$ is fixed, we obtain a coupling between $\langle x,g \rangle$ and $G$ by identifying $g_1$ with $g_2$.
It follows that
\begin{align*}
W_2(\langle x,g \rangle , G) & \leq \E_{x \sim p} \left(\norm{x}_2 - \sqrt{n} \right)^2 \cdot \E_{g_1 \sim \Gauss(0,1)} g_1^2 \\
& = \E_{x \sim p} \left(\norm{x}_2 - \sqrt{n} \right)^2 \\
& \leq   \E_{x \sim p} \left(\frac{ \left(\norm{x}_2^2 - n \right)^2}{\left(\norm{x}_2 + \sqrt{n}\right)^2} \right) \\
&\leq   \frac{1}{n} \cdot \Var \left(\norm{x}_2^2 \right) \\
&\leq   \frac{1}{n} \cdot O \left(\psi_n^2 \right) ~=~ O \left(n^{1-2\epsilon} \right),
\end{align*}
where the last line uses Lemma~\ref{lem:quadratic-form} with the matrix $A$ being the identity matrix in $\Rn$.
This combined with~(\ref{eqn:ytoGauss}) finishes the proof of Theorem~\ref{thm:KLStoGenCLT}.
\end{proofof}

\eat{
If for $n$-dimensional isotropic log-concave distributions $p$ and $q$, we denote $\langle p,q \rangle$ as the inner product between independent random vectors drawn uniformly from $p$ and $q$, then Lemma~\ref{lem:InnerProdlog-concaveGauss} allows us to relate $\langle p,q \rangle$ to $\langle p,\Gauss(0,I) \rangle$ and also $ \langle \Gauss(0,I), p \rangle$ to $\langle \Gauss(0,I),\Gauss(0,I) \rangle$.
Since $\langle \Gauss(0,I),\Gauss(0,I) \rangle$ can be written as the sum of $n$ i.i.d. random variable $\Gauss(0,1) \cdot \Gauss(0,1)$, an application of standard 1-dimensional central limit theorem relates
$\langle \Gauss(0,I),\Gauss(0,I) \rangle$ to $\Gauss(0,n)$.

\begin{proofof}{Theorem~\ref{thm:KLStoGenCLT} Using Lemma~\ref{lem:InnerProdlog-concaveGauss}}
Let $g,g^1,g^2$ be i.i.d. random vectors from standard $n$-dimensional normal distribution $\Gauss(0,I)$ and $x$ be an independent random vector drawn from $p$.

By Lemma~\ref{lem:InnerProdlog-concaveGauss}, we have
\begin{align}
\label{eqn:ytoGauss}
W_2 (\langle x,y \rangle, \langle x,g \rangle )^2 = O(n^{1-2\epsilon})
\end{align}
and 
\begin{align}
\label{eqn:xtoGauss}
W_2 (\langle z,g \rangle, \langle g^1,g^2 \rangle )^2 = O(n^{1-2\epsilon})
\end{align}
It follows that
\begin{align*}
W_2 (\langle x,y \rangle, \langle g^1,g^2 \rangle )^2 = O(n^{1-2\epsilon})
\end{align*}
Notice $\langle g^1,g^2 \rangle = \sum_{i \in [n]} g^1_i g^2_i$, where $g^1_i$'s and $g^2_i$'s are i.i.d samples from $\Gauss(0,1)$.
Notice the fourth moment $\E[(g^1_ig^2_i)^4]$ is bounded by some constant, the result in~\cite{bonis2015rates} implies that the $W_2$ distance between $\frac{1}{\sqrt{n}} \langle g^1,g^2 \rangle$ and a standard normal distribution is bounded by $O(1/\sqrt{n})$. 
It follows that
\begin{align*}
W_2(\langle g^1,g^2 \rangle,G) = O(1)
\end{align*}
This combined with Equation (\ref{eqn:ytoGauss}) and (\ref{eqn:xtoGauss}) finishes the proof of Theorem~\ref{thm:KLStoGenCLT}.  
\end{proofof}
}

Now we are left to prove Lemma~\ref{lem:InnerProdlog-concaveGauss}.
For this we turn to the stochastic localization technique introduced in Section~\ref{sec:StochasticLocalization}.
In the proof, we make use of Lemma \ref{lem:BoundOperNorm}. 
Our proof here bears structural similarities to that in \cite{eldan2018clt}, in that both proofs use stochastic localization specifically by viewing random variables as Brownian motion.

\begin{proofof}{Lemma~\ref{lem:InnerProdlog-concaveGauss}}
We apply the stochastic construction in Section~\ref{sec:StochasticLocalization} with initial probability distribution $p_0 = p$. 
Since $p_t$ is a martingale and $p_\infty$ is a point mass at $\mu_\infty$, we have that
\begin{align*}
x \sim \mu_\infty = \int_0^\infty \dd \mu_t = \int_0^\infty A_t \dd W_t^{(n)},
\end{align*}
where we used Lemma \ref{lem:dA} and $W_t^{(n)}$ is a standard $n$-dimensional Brownian motion. 
The inner product $\langle x,y \rangle$ can be written similarly as
\begin{align*}
\langle x,y \rangle = \int_0^\infty y^T A_t \dd W_t^{(n)}.
\end{align*}
Notice that $y^T A_t \dd W_t^{(n)}$ is a martingale whose quadratic variation has derivative $y^T A_t^2 y$ at time $t$ .
It follows that the process $W_t^{(1)}$ defined by $\dd W_t^{(1)} = y^T A_t \dd W_t^{(n)}/ \sqrt{y^T A_t^2 y}$ is a 1-dimensional standard Brownian motion.
We therefore have
\begin{align*}
\langle x,y \rangle = \int_0^\infty \sqrt{y^T A_t^2 y} \cdot \dd W_t^{(1)}.
\end{align*}
Note that $\sqrt{y^T A_t^2 y}$ is concentrated near $\sqrt{\E_{y\sim q} y^T A_t^2 y} = \sqrt{\tr\left(A_t^2 \right)}$.
It is therefore natural to couple $\langle x,y \rangle$ with the random variable $L = \int_0^\infty \sqrt{\tr\left(A_t^2 \right)} \dd W_t^{(1)}$. 
We will show that this coupling gives an upper bound on $W_2(\langle x,y \rangle,L)^2$. 
Notice that the first random variable $\langle x,y \rangle $ depends on both $x$ and $y$ but the second random variable $L$ depends only on $x$. 
So why would this coupling work?
The intuition behind the coupling is the following: as one takes the expectation over $y$, the random variable $\sqrt{y^T A_t^2 y}$ is concentrated around $\sqrt{\tr \left(A_t^2 \right)}$ and the deviation depends on the variable $\norm{A_t}_{\op}^2$.
In the stochastic construction in Section~\ref{sec:StochasticLocalization}, $A_t$ starts from identity and ends up being $0$.
This allows good bounds on $\norm{A_t}_{\op}^2$.

We use $\E_{x}$ to denote the expectation taken with respect to the randomness of $W_t^{(n)}$ (notice that both $A_t$ and $W_t^{(1)}$ adapt to $W_t^{(n)}$). It follows that
\begin{align*}
W_2(\langle x,y \rangle,L)^2 &\leq \E_{x,y}\left[ \int_0^\infty \left(\sqrt{y^T A_t^2 y} - \sqrt{\tr\left(A_t^2\right)} \right) \cdot \dd W_t^{(1)} \right]^2 \\
& = \E_{x,y} \left[ \int_0^\infty \left(\sqrt{y^T A_t^2 y} - \sqrt{\tr\left(A_t^2\right)} \right)^2 \dd t \right] \\
& =  \int_0^\infty \E_{x,y} \left[ \left(\sqrt{y^T A_t^2 y} - \sqrt{\tr\left(A_t^2\right)} \right)^2 \right] \dd t \\
& = \int_0^\infty \E_{x,y} \left[ \left(\frac{y^T A_t^2 y - \tr\left(A_t^2\right)}{\sqrt{y^T A_t^2 y} + \sqrt{\tr\left(A_t^2\right)} } \right)^2 \right] \dd t \\
& \leq \int_0^\infty \E_{x} \left[ \frac{ \E_y \left(y^T A_t^2 y - \tr\left(A_t^2\right) \right)^2}{\tr\left(A_t^2\right) } \right] \dd t  \\
& = \int_0^\infty \E_{x} \left[ \frac{ \Var \left(y^T A_t^2 y \right)}{\tr\left(A_t^2\right) } \right] \dd t  \\
&\leq \int_0^\infty \E_{x} \left[ \frac{O\left(\psi_n^2 \right) \cdot  \tr\left(A_t^4\right)}{\tr\left(A_t^2 \right) } \right] \dd t \\
& \leq O\left(\psi_n^2 \right) \cdot  \int_0^\infty   \E_{x} \left[ ||A_t||_\op^2 \right] \dd t,
\end{align*}
where the first equality uses Ito's isometry and the last two lines follow from Lemma~\ref{lem:quadratic-form}.
The remaining thing is to bound $||A_t||_\op^2$.

The convariance matrix $A_t$ corresponds to a density proportional to the log-concave density $p(x)$ multiplied by a Gaussian density $e^{-c_t^T x - \frac{t}{2}||x||_2^2}$. 
It is well known that the operator norm of such $A_t$ is dominated by the Gaussian term (e.g.~\cite{Eldan2013}, Proposition 2.6), i.e.
\[
||A_t||_\op \leq O(1/t).
\]
We also need an upper bound for $\E_x[||A_t||_\op^2]$ when $t$ is close to 0. For this
take $k=\frac{1}{1/2-\epsilon}$ in Lemma \ref{lem:BoundOperNorm}, we have for any $0 \leq t \leq \frac{1/2-\epsilon}{c n^{1/2-\epsilon}(\log n)^{1/2 + \epsilon}}$,
\begin{align}
\p \left[||A_t||_\op \geq 2 \right] \leq 2 \exp\left( -\frac{1}{ct} \right).
\end{align}
We can therefore bound $\E[||A_t||_\op^2]$ as
\begin{align*}
\E[||A_t||_\op^2] \leq 4 \cdot \p \left[||A_t||_\op < 2 \right] + \frac{1}{t^2} \cdot \p \left[||A_t||_\op \geq 2 \right]
\leq 4 + \frac{1}{t^2} \cdot 2 \exp\left( -\frac{1}{ct} \right).
\end{align*}
Since $t\leq \frac{1/2-\epsilon}{c n^{1/2-\epsilon }(\log n)^{1/2 + \epsilon}}$, $1/t \geq \frac{c n^{1/2-\epsilon}(\log n)^{1/2 + \epsilon}}{1/2-\epsilon}$.
For fixed $0<\epsilon<1/2$, the last term $\frac{1}{t^2} \cdot 2 \exp\left( -\frac{1}{ct} \right)$ becomes negligible when $n$ is sufficiently large so $\E[||A_t||_\op^2]$ is bounded by some constant $C_\epsilon$ (that depends on $\epsilon$) for any $t \leq \frac{1/2-\epsilon}{c n^{1/2-\epsilon}(\log n)^{1/2 + \epsilon}} 
=T \leq 1$.
It follows that
\begin{align*}
 W_2(\langle x,y \rangle, L)^2 &\leq O\left(\psi_n^2 \right) \cdot \int_0^\infty   \E_{x} \left[ ||A_t||_\op^2 \right] \dd t \\
&\leq O\left(\psi_n^2 \right) \cdot \left( \int_0^T C_\epsilon \dd t + \int_T^\infty \frac{1}{t^2} \dd t \right)\\
& \leq O\left(\psi_n^2 \right) \cdot \frac{1}{T} = O\left(n^{1-2\epsilon}(\log n)^{1/2 + \epsilon} \right).
\end{align*}

We note that $L$ is defined using only the isotropic log-concave distribution $p$. One can therefore prove a similar bound when $q$ is the $n$-dimensional standard normal distribution, i.e.
\begin{align*}
W_2(\langle x,g \rangle, L)^2 =O \left(n^{1-2\epsilon}(\log n)^{1/2 + \epsilon} \right).
\end{align*}
Combining these two bounds, we have the desired result.
\begin{align*}
W_2 (\langle x,y \rangle, \langle x,g \rangle )^2 = O\left(n^{1-2\epsilon}(\log n)^{1/2 + \epsilon} \right).
\end{align*}
\end{proofof}

\subsection{Connection to Classical CLT for Convex Sets}
Using exactly the same approach, we prove the following theorem which is easier to compare with classical results on central limit theorem for convex sets.
Here we replace the $W_2$ distance in Theorem~\ref{eqn:ExpgeneralizedCLTW2} by the total variation distance.
\begin{thm}
\label{thm:LinktoClassicalCLT}
Assume $\psi_n = O\left(n^{1/4-\epsilon/2} \right)$ for some $0<\epsilon<1/2$ and some dimension $n$. Let $p,q$ be any isotropic log-concave distributions in $\Rn$. 
For fixed vector $x \sim p$, denote $\langle x,y \rangle$ the random variable formed by the inner product of $x$ and $y$, when $y \sim q$ is independently drawn from $x$.
Let $g \sim \Gauss(0,1)$ be a standard normal distribution.
Then we have
\[
\P_{x \sim p} \left[ \dtv \left( \frac{\langle x,y \rangle}{\norm{x}_2} , g  \right) \geq C n^{-\epsilon/2} \right] \leq \exp \left(-c n^{\frac{1}{2}-\epsilon} (\log n)^{1/2 + \epsilon} \right),
\]
for some constants $c$ and $C$ that depend on $\epsilon$.
\end{thm}

The following lemma can be proved by using a similar approach as in the proof of Lemma~\ref{lem:InnerProdlog-concaveGauss}.

\begin{lem}
\label{lem:W2HighProbBound}
Assume $\psi_n = O\left(n^{1/4-\epsilon/2} \right)$ for some $0 < \epsilon < 1/2$ and some dimension $n$.
Let $p,q$ be any isotropic log-concave distributions in $\Rn$ and let $x\sim p$, $y\sim q$ and $g\sim \Gauss(0,I)$ be independent samples.
Then with probability at least $1-\exp \left(- c n^{\frac{1}{2}-\epsilon} (\log n)^{1/2 + \epsilon} \right)$
over the random choice of $x$, we have
\[
W_2(\langle x,y \rangle, \langle x,g \rangle) = O \left ( n^{\frac{1}{2} - \epsilon} \right),
\]
where the constant $c$ depends on $\epsilon$.
\end{lem}

\begin{proofof}{Theorem~\ref{thm:LinktoClassicalCLT} Using  Lemma~\ref{lem:W2HighProbBound}}
By Lemma~\ref{lem:small-ball}, we have with probability at least $1 - \exp(-\Omega(\sqrt{n}))$, $\norm{x}_2 \geq C \sqrt{n}$ for some universal constant $C>0$.
We condition on this event and the event in Lemma~\ref{lem:W2HighProbBound} such that
\[
W_2(\langle x,y \rangle, \langle x,g \rangle) = O \left ( n^{\frac{1}{2} - \epsilon} \right).
\]
The probability that these events hold at the same time is at least 
\[
1-\exp \left(-\Omega \left(n^{\frac{1}{2}-\epsilon} (\log n)^{1/2 + \epsilon} \right) \right).
\]
In this case we have
\[
W_2 \left(\langle x,y \rangle/\norm{x}_2, \langle x,g \rangle/\norm{x}_2 \right) = O \left ( n^{- \epsilon} \right).
\]
Notice that for a fixed $x$, $\langle x,y \rangle/\norm{x}_2$ follows a 1-dimensional isotropic log-concave distribution and $\langle x,g \rangle/\norm{x}_2$ follows a standard normal distribution.
Applying Lemma~\ref{lem:TVtoWasser} finishes the proof of the theorem.
\end{proofof}

\appendix

\section{Missing Proofs in Section~\ref{subsec:DistProb}}
\label{sec:missProofsSecDistProb}
We restate Lemma~\ref{lem:WqboundbyWp} below for reference.

\WqboundbyWp*
\begin{proofof}{Lemma~\ref{lem:WqboundbyWp}}
The result for Case 1 is given by~\cite[Prop 5]{meckes2014equivalence}.
Here we use the same idea to prove the result for Case 2. 
The proof for Case 3 is almost the same and is omitted.

We denote the random variable drawn from $\nu$ as $z$ and the best coupling for $W_s(\mu,\nu)$ as $\left(\frac{1}{\sqrt{n}}\langle x,y \rangle, z \right)$.
We use the coupling $\left(\frac{1}{\sqrt{n}}\langle x,y \rangle, z \right)$ in the rest of the proof whenever we write expectations.
Denote $\Ind_{\{ \cdot \}}$ the indicator function of an event. 
For any $R>0$, we have
\begin{align*}
W_t\left(\frac{1}{\sqrt{n}}\langle x,y \rangle, z \right)^t
&\leq \E \left| \frac{1}{\sqrt{n}}\langle x,y \rangle - z \right|^t \\
&\leq R^{t-s} \cdot \E \left| \frac{1}{\sqrt{n}}\langle x,y \rangle - z \right|^s + \E \left| \frac{1}{\sqrt{n}}\langle x,y \rangle - z \right|^t \Ind_{ \left\{\left| \frac{1}{\sqrt{n}}\langle x,y \rangle - z \right| \geq R \right\}} \\
& \leq R^{t-s} \cdot W_s\left(\frac{1}{\sqrt{n}}\langle x,y \rangle, z \right)^s + \sqrt{\p\left[ \left| \frac{1}{\sqrt{n}}\langle x,y \rangle - z \right| \geq R \right] \cdot \E \left(\frac{1}{\sqrt{n}}\langle x,y \rangle - z \right)^{2t} } ,
\end{align*}
where the last step is by Cauchy-Schwarz.
Now we bound the second term in the above expression. 
Using Minkowski's inequality, we have
\begin{align*}
\left(\E \left(\frac{1}{\sqrt{n}}\langle x,y \rangle - z \right)^{2t} \right)^{1/2t} 
&\leq \left( \E z ^{2t} \right)^{1/2t} + \left( \E \left( \frac{1}{\sqrt{n}}\langle x,y \rangle \right)^{2t} \right)^{1/2t}.
\end{align*}
Since $z$ follows an isotropic log-concave distribution, it follows from Lemma~\ref{lem:lcmom} that $\left( \E z^{2t} \right)^{1/2t} \leq 4t$.
For the second term we notice that when $x$ is fixed, the random variable $\frac{1}{\sqrt{n}} \langle x,y \rangle$ follows a 1-dimensional log-concave distribution with variance $\frac{\norm{x}_2^2}{n}$.
Using Lemma~\ref{lem:lcmom} again, we have
\begin{align*}
\E \left( \frac{1}{\sqrt{n}}\langle x,y \rangle \right)^{2t} 
& \leq \E_{x \sim p} \left( \E_{y \sim q}\left( \frac{1}{\sqrt{n}}\langle x,y \rangle \right)^{2t} \right)\\
& \leq  (4t)^{2t} \cdot  \E_{x \sim p} \frac{\norm{x}_2^{2t}}{n^t} ~\leq~ (4t)^{4t}.
\end{align*}
We therefore have
\[
\E \left(\frac{1}{\sqrt{n}}\langle x,y \rangle - z \right)^{2t} \leq \left(4t + 16t^2 \right)^{2t}.
\]

Now we bound $\p\left[ \left| \frac{1}{\sqrt{n}}\langle x,y \rangle - z \right| \geq R \right]$ as follows.
For some constant $c_2,C_R > 0$, whenever $R >C_R$ we have
\begin{align*}
\p\left[ \left| \frac{1}{\sqrt{n}}\langle x,y \rangle - z \right| \geq R \right] 
&\leq \p\left[ \left| \frac{1}{\sqrt{n}}\langle x,y \rangle\right| \geq R/2 \right] + \p\left[ \left| z \right| \geq R/2 \right]\\
& \leq \p\left[ \left| \frac{1}{\sqrt{n}}\langle x,y \rangle\right| \geq R/2 \right] + \exp(-c_2 R).
\end{align*}

Since $x$ follows an isotropic log-concave distribution, we have from Theorem \ref{thm:Paouris} that whenever $R > C_R$, there exist constants $c_1,C>0$ such that 
\[
\p[\norm{x}_2 \geq \sqrt{Cn}] \leq \exp(-c_1\sqrt{ n}).
\]
Whenever $\norm{x}_2 < \sqrt{Cn}$ for fixed vector $x$, the random variable $\frac{1}{\sqrt{n}}\langle x,y \rangle$ follows a 1-dimensional log-concave distribution with variance at most $C$.
Therefore when the universal constant $C_R$ is large enough and when $R>C_R$, we have
\begin{align*}
\p\left[ \left| \frac{1}{\sqrt{n}}\langle x,y \rangle\right| \geq R/2 \right] 
&\leq \exp(-c_1 \sqrt{n}) + \exp(-c_2 R).
\end{align*}

Combining everything we have that when $R > C_R$,
\begin{align*}
W_t\left(\frac{1}{\sqrt{n}}\langle x,y \rangle, z \right)^t \leq R^{t-s} W_s\left(\frac{1}{\sqrt{n}}\langle x,y \rangle, z \right)^s + (4t + 16t^2)^t \cdot \sqrt{2\left(\exp(-c_2 R) + \exp(-c_1 \sqrt{n})\right)}.
\end{align*}
Optimizing over $R$, for some constant $c \geq 0$ we have
\[
W_t\left(\frac{1}{\sqrt{n}}\langle x,y \rangle, z \right)^t \leq c \cdot W_s\left(\frac{1}{\sqrt{n}}\langle x,y \rangle, z \right)^s \cdot \log^{t-s}\left( \frac{c^t t^{2t}}{W_s\left(\frac{1}{\sqrt{n}}\langle x,y \rangle, z \right)^s} \right)  + c^t t^{2t}\exp(-c \sqrt{n}).
\]
This finishes the proof of Lemma~\ref{lem:WqboundbyWp}.
\end{proofof}

\section{Missing Proofs in Section~\ref{subsubsec:TensorBounds}}
\label{sec:missProofsTensorBounds}
In this section, we give proofs of the lemmas in Section~\ref{subsubsec:TensorBounds}.
Here we repeatedly use the elementary facts that $\tr(AB)=\tr(BA)$ and $x^{T}Ay=\tr\left(Ayx^{T}\right)$.

\begin{lem}
\label{lem:tequ}For any isotropic log-concave distribution $p$ and
symmetric matrices $A$ and $B$, we have that
\[
T_{p}(A,B,I)=\sum_{i}\tr(A\Delta_{i}B\Delta_{i})
\quad \text{and} \quad
T_{p}(A,B,I)=\sum_{i,j}A_{ij}\tr(\Delta_{i}B\Delta_{j}),
\]
where $\Delta_{i}=\E_{x\sim p}xx^{T}x_{i}$.
\end{lem}
\begin{proof}
Direct calculation shows that
\begin{align*}
T_{p}(A,B,I) & =\E_{x,y\sim p}x^{T}Ayx^{T}Byx^{T}y=\sum_{i}\E_{x,y\sim p}x^{T}Ayx^{T}Byx_{i}y_{i}\\
 & =\sum_{i}\E_{x,y\sim p}\tr\left(Axx^{T}Byy^{T}x_{i}y_{i}\right)=\sum_{i}\tr(A\Delta_{i}B\Delta_{i}),
\end{align*}
and
\begin{align*}
T_{p}(A,B,I) & =\E_{x,y\sim p}x^{T}Ayx^{T}Byx^{T}y=\sum_{i,j}A_{ij}\E_{x,y\sim p}x_{i}y_{j}x^{T}Byx^{T}y\\
 & =\sum_{i,j}A_{ij}\E_{x,y\sim p}\tr\left(xx^{T}Byy^{T}x_{i}y_{j}\right)=\sum_{i,j}A_{ij}\tr(\Delta_{i}B\Delta_{j}).
\end{align*}
\end{proof}

\trabs*

\begin{proof}
Fix any isotropic log-concave distribution $p$. We define $\Delta_{i}=\E_{x\sim p}xx^{T}x^{T}A_{3}^{1/2}e_{i}$
which is well defined since $A_{3}\succeq0$. Then, we have that
\begin{align*}
T_{p}(A_{1},A_{2},A_{3})=  \E_{x,y\sim p}x^{T}A_{1}yx^{T}A_{2}yx^{T}A_{3}y
=  \sum_{i}\tr(A_{1}\Delta_{i}A_{2}\Delta_{i}).
\end{align*}
Since $\Delta_{i}$ is symmetric and $A_{1},A_{2}\succeq0$, we have
that $A_{1}^{1/2}\Delta_{i}A_{2}\Delta_{i}A_{1}^{1/2}\succeq0$ and
$\tr(A_{1}\Delta_{i}A_{2}\Delta_{i})\geq0$. Therefore, $T(A_{1},A_{2},A_{3})\geq T_{p}(A_{1},A_{2},A_{3})\geq0$.

For the second part, we write $B_{1}=B_{1}^{(1)}-B_{1}^{(2)}$ where
$B_{1}^{(1)}\succeq0$, $B_{1}^{(2)}\succeq0$ and $\left|B_{1}\right|=B_{1}^{(1)}+B_{1}^{(2)}$.
We define $B_{2}^{(1)},B_{2}^{(2)},B_{3}^{(1)},B_{3}^{(2)}$ similarly.
Note that
\begin{align*}
T(B_{1},B_{2},B_{3})= & T\left(B_{1}^{(1)},B_{2}^{(1)},B_{3}^{(1)} \right)-T \left(B_{1}^{(1)},B_{2}^{(1)},B_{3}^{(2)} \right)-T\left(B_{1}^{(1)},B_{2}^{(2)},B_{3}^{(1)} \right)+T\left(B_{1}^{(1)},B_{2}^{(2)},B_{3}^{(2)} \right)\\
 & -T\left(B_{1}^{(2)},B_{2}^{(1)},B_{3}^{(1)} \right)+T\left(B_{1}^{(2)},B_{2}^{(1)},B_{3}^{(2)}\right)+T\left(B_{1}^{(2)},B_{2}^{(2)},B_{3}^{(1)}\right)-T\left(B_{1}^{(2)},B_{2}^{(2)},B_{3}^{(2)} \right).
\end{align*}
Since $B_{j}^{(i)}\succeq0$, the first part of this lemma shows that
every term $T\left(B_{1}^{(i)},B_{2}^{(j)},B_{3}^{(k)} \right)\geq0$. Hence,
we have that
\begin{align*}
T\left(B_{1},B_{2},B_{3}\right)\leq & T\left(B_{1}^{(1)},B_{2}^{(1)},B_{3}^{(1)}\right)+T\left(B_{1}^{(1)},B_{2}^{(1)},B_{3}^{(2)}\right)+T\left(B_{1}^{(1)},B_{2}^{(2)},B_{3}^{(1)}\right)+T\left(B_{1}^{(1)},B_{2}^{(2)},B_{3}^{(2)}\right)\\
 & +T\left(B_{1}^{(2)},B_{2}^{(1)},B_{3}^{(1)}\right)+T\left(B_{1}^{(2)},B_{2}^{(1)},B_{3}^{(2)}\right)+T\left(B_{1}^{(2)},B_{2}^{(2)},B_{3}^{(1)}\right)+T\left(B_{1}^{(2)},B_{2}^{(2)},B_{3}^{(2)}\right)\\
= & T\left(\left|B_{1}\right|,\left|B_{2}\right|,\left|B_{3}\right| \right).
\end{align*}
\end{proof}

\begin{lem}
\label{lem:trDAD}Suppose that $\psi_{k}\leq\alpha k^{\beta}$ for
all $k\leq n$ for some $0\leq\beta\leq\frac{1}{2}$ and $\alpha\geq1$.
Given an isotropic log-concave distribution $p$ and a unit vector
$v$, the following two statements hold for $\Delta=\E_{x\sim p}xx^{T}x^{T}v$:
\begin{enumerate}
\item For any orthogonal projection matrix $P$ with rank $r$, we have
that
\[
\tr(\Delta P\Delta)\leq O \left(\psi_{\min(2r,n)}^{2} \right).
\]
\item For any symmetric matrix $A$, we have that 
\[
\tr(\Delta A\Delta)\leq O \left(\alpha^{2}\log n \right) \cdot \left(\tr\left|A\right|^{1/(2\beta)}\right)^{2\beta}.
\]
\end{enumerate}
\end{lem}
\begin{proof}
We first bound $\tr(\Delta P\Delta)$. This part of the proof is generalized
from a proof by Eldan \cite{Eldan2013}. Note that $\tr(\Delta P\Delta)=\E_{x\sim p}x^{T}P\Delta xx^{T}v.$
Since $\E x^{T}v=0$, we have that
\begin{align*}
\tr(\Delta P\Delta) 
\leq 
\sqrt{\E\left(x^{T}v \right)^{2}}\sqrt{\Var\left(x^{T}P\Delta x\right)}
\overset{\lref{quadratic-form}}{\leq}
O\left(\psi_{\rank(P\Delta+\Delta P)}\right) \cdot \sqrt{\tr\left(\Delta P\Delta\right)}.
\end{align*}
This gives $\tr(\Delta P\Delta)\leq O\left(\psi_{\min(2r,n)}^{2} \right)$.

Now we bound $\tr(\Delta A\Delta)$. Since $\tr(\Delta A\Delta)\leq\tr(\Delta\left|A\right|\Delta)$,
we can assume without loss of generality that $A\succeq0$. We write
$A=\sum_{i}A_{i}+B$ where each $A_{i}$ has eigenvalues between $\big(\norm {A}_\op2^{i}/n,\norm {A}_\op2^{i+1}/n \big]$
and $B$ has eigenvalues smaller than or equals to $\norm {A}_\op/n$.
Clearly, we only need at most $\left\lceil \log(n)+1\right\rceil $
many such $A_{i}$. Let $P_{i}$ be the orthogonal projection from
$\Rn$ to the span of the range of $A_{i}$. Using $\norm{A_{i}}_{\op}P_{i}\succeq A_{i}$,
we have that 
\[
\tr(\Delta A_{i}\Delta)\leq\norm{A_{i}}_{\op}\tr(\Delta P_{i}\Delta)\leq O\left(\psi_{\min(2\rank(A_{i}),n)}^{2}\right) \cdot \norm{A_{i}}_{\op}\leq O(\alpha^{2}) \cdot \sum_{i}\rank(A_{i})^{2\beta}\norm{A_{i}}_{\op},
\]
where we used the first part of this lemma in the last inequality.

Similarly, we have that
\[
\tr(\Delta B\Delta)\leq O\left(\psi_{n}^{2}\right) \cdot \norm {B}_{\op}\leq O(n\norm {B}_{\op})\leq O(1) \cdot \norm {A}_\op.
\]

Combining the bounds on $\tr(\Delta A_{i}\Delta)$ and $\tr(\Delta B\Delta)$,
we have that
\begin{align*}
\tr(\Delta A\Delta) & \leq O(\alpha^{2}) \cdot \sum_{i}\rank(A_{i})^{2\beta}\norm{A_{i}}_{\op}+O(1) \cdot \norm {A}_\op\\
 & \leq O(\alpha^{2}) \cdot \left(\sum_{i}\rank(A_{i})\norm{A_{i}}_{\op}^{1/(2\beta)}\right)^{2\beta}\log(n)^{1-2\beta}\\
 & \leq O(\alpha^{2}\log n) \cdot \left(\tr\left|A\right|^{1/(2\beta)}\right)^{2\beta}.
\end{align*}
\end{proof}

In the next lemma, we collect tensor inequalities that will
be useful for later proofs.

\tinq*

\begin{proof}
Without loss of generality, we can assume $A$ is diagonal by rotating
space. In particular, if we want to prove something for $\tr(A^{\alpha}\Delta A^{\beta}\Delta)$
where $A,\Delta$ are symmetric matrices, we use the spectral decomposition
$A=U\Sigma U^{T}$ to rewrite this as 
\[
\tr\left(U\Sigma^{\alpha}U^{T}\Delta U\Sigma^{\beta}U^{T}\Delta\right)=\tr\left(\Sigma^{\alpha}\left(U^{T}\Delta U \right)\Sigma^{\beta} \left(U^{T}\Delta U\right)\right),
\]
which puts us back in the same situation, but with a diagonal matrix
$A$. 
%
%
%
%
For all inequalities listed above, it suffices to upper bound $T$ by upper
bounding $T_{p}$ for any isotropic log-concave distribution $p$.

For inequality \ref{lem:tinq1}, we note that

\[
T_{p}(A,I,I)\overset{\lref{tequ}}{=}\sum_{i}A_{ii}\tr(\Delta_{i}^{2})\leq\norm {A}_\op\sum_{i}\tr\left(\Delta_{i}^{2} \right)\overset{\lref{tequ}}{=}\norm {A}_\op T(I,I,I),
\]
where the last inequality is from the third moment assumption.

For inequality \ref{lem:tinq2}, we note that
\[
T_{p}(A,I,I)\overset{\lref{tequ}}{=}\sum_{i}A_{ii}\tr(\Delta_{i}^{2})\overset{\lref{trDAD}}{\leq}\sum_{i}\left|A_{ii}\right| \cdot O\left(\psi_{n}^{2} \right)=O\left(\psi_{n}^{2} \right) \cdot \tr\left|A\right|.
\]

For inequality \ref{lem:tinq3}, we let $P$ be the orthogonal projection
from $\Rn$ to the span of the range of $B$. Then,
we have that
\begin{align*}
T_{p}(A,B,I) & \leq T_{p}(|A|,|B|,I)\ttag{\lref{trabs}}\\
 & =\sum_{i}\left|A_{ii}\right|\tr(\Delta_{i}|B|\Delta_{i})\ttag{\lref{tequ}}\\
 & \overset{\cirt 1}{\leq}\norm {B}_{\op}\sum_{i}\left|A_{ii}\right|\tr(\Delta_{i}P\Delta_{i})\\
 & \leq O\left(\psi_{r}^{2} \right) \cdot \tr|A|\norm {B}_{\op}.\ttag{\lref{trDAD}}
\end{align*}
where we used that $\left|B\right| \preceq \norm {B}_{\op}P$ in $\cirt 1$.

For inequality \ref{lem:tinq4}, we note that
\[
T_{p}(A,B,I)\overset{\lref{tequ}}{=}\sum_{i}A_{ii}\tr(\Delta_{i}B\Delta_{i})\overset{\lref{trDAD}}{\leq}O(\alpha^{2} \log n) \cdot \tr\left|A\right|\left(\tr\left|B\right|^{1/(2\beta)}\right)^{2\beta}.
\]

For inequality \ref{lem:tinq5}, we note that
\begin{align*}
T_{p}(A,B,I) & \leq T_{p}(\left|A\right|,\left|B\right|,I)\ttag{\lref{trabs}}\\
 & =\sum_{i}\tr(\left|A\right|\Delta_{i}\left|B\right|\Delta_{i})\ttag{\lref{tequ}}\\
 & \leq\sum_{i}\tr(\left|A\right|\left|\Delta_{i}\right|\left|B\right|\left|\Delta_{i}\right|)\\
 & =\sum_{i}\tr\left(\left|\Delta_{i}\right|^{1/s}\left|A\right|\left|\Delta_{i}\right|^{1/s}\left|\Delta_{i}\right|^{1/t}\left|B\right|\left|\Delta_{i}\right|^{1/t} \right)\\
 & \leq\sum_{i}\left(\tr\left(\left(\left|\Delta_{i}\right|^{1/s}\left|A\right|\left|\Delta_{i}\right|^{1/s}\right)^{s}\right)\right)^{1/s} \cdot \left(\tr\left(\left(\left|\Delta_{i}\right|^{1/t}\left|B\right|\left|\Delta_{i}\right|^{1/t}\right)^{t}\right)\right)^{1/t}\ttag{\lref{matrixholder}}\\
 & \leq\sum_{i}\left(\tr\left(\left|\Delta_{i}\right|\left|A\right|^{s}\left|\Delta_{i}\right|\right)\right)^{1/s} \cdot \left(\tr\left(\left|\Delta_{i}\right|\left|B\right|^{t}\left|\Delta_{i}\right|\right)\right)^{1/t}\ttag{\lref{lieborg}}\\
 & =\sum_{i}\left(\tr\left(\left|A\right|^{s}\Delta_{i}^{2}\right)\right)^{1/s}\cdot \left(\tr\left(\left|B\right|^{t}\Delta_{i}^{2}\right)\right)^{1/t}\\
 & \leq\left(\sum_{i}\tr\left(\left|A\right|^{s}\Delta_{i}^{2}\right)\right)^{1/s} \cdot \left(\sum_{i}\tr\left(\left|B\right|^{t}\Delta_{i}^{2}\right)\right)^{1/t}\\
 & =\left(T_{p}\left(\left|A\right|^{s},I,I \right)\right)^{1/s} \cdot \left(T_{p}\left(\left|B\right|^{t},I,I\right)\right)^{1/t}.\ttag{\lref{tequ}}
\end{align*}
\end{proof}

\liebtr*

\begin{proof}
Fix any isotropic log-concave distribution $p$. Let $\Delta_{i}=\E_{x\sim p}B^{1/2}xx^{T}B^{1/2}x^{T}C^{1/2}e_{i}$.
Then, we have that
\begin{align*}
& T_{p}(B^{1/2}A^{\alpha}B^{1/2},B^{1/2}A^{1-\alpha}B^{1/2},C)\\ & =\E_{x,y\sim p}x^{T}B^{1/2}A^{\alpha}B^{1/2}yx^{T}B^{1/2}A^{1-\alpha}B^{1/2}yx^{T}Cy\\
 & =\sum_{i}\E\left(\left(y^{T}B^{1/2}A^{\alpha}B^{1/2}x\right)\left(x^{T}B^{1/2}A^{1-\alpha}B^{1/2}y\right)x^{T}C^{1/2}e_{i}y^{T}C^{1/2}e_{i}\right)\\
 & =\sum_{i}\E\left(\tr\left(A^{\alpha}B^{1/2}xx^{T}B^{1/2}A^{1-\alpha}B^{1/2}yy^{T}B^{1/2}\right)\left(x^{T}C^{1/2}e_{i}\right)\left(y^{T}C^{1/2}e_{i}\right)\right)\\
 & =\sum_{i}\tr(A^{\alpha}\Delta_{i}A^{1-\alpha}\Delta_{i}).
\end{align*}
Using Lemma \ref{lem:lieb}, we have that
\[
\sum_{i}\tr\left(A^{\alpha}\Delta_{i}A^{1-\alpha}\Delta_{i} \right)\leq\sum_{i}\tr\left(A\Delta_{i}^{2}\right)=\E_{x,y\sim p}x^{T}B^{1/2}AB^{1/2}yx^{T}Byx^{T}Cy=T_{p}\left(B^{1/2}AB^{1/2},B,C \right).
\]
Taking the supremum over all isotropic log-concave distributions,
we get the result.
\end{proof}

\bibliographystyle{plain}
\bibliography{bib}

\end{document}